\documentclass[journal]{IEEEtran}

\IEEEoverridecommandlockouts                              

\usepackage{balance}

\usepackage{graphicx}
\usepackage{mathptmx} 
\usepackage{times} 
\usepackage{amsmath} 
\usepackage{amsthm}
\usepackage{amssymb}  
\usepackage{amsfonts}
\usepackage{newtxmath}

\usepackage{mathtools, cuted}

\usepackage{bm}

\usepackage{enumerate}

\usepackage[utf8]{inputenc}

\usepackage{tikz,tabstackengine,amsmath}
\usetikzlibrary{automata,positioning, arrows.meta}

\usepackage[noadjust]{cite}

\newtheorem{theorem}{Theorem}[section]
\newtheorem{lemma}[theorem]{Lemma}

\newtheorem{corollary}[theorem]{Corollary}

\theoremstyle{remark}

\theoremstyle{definition}
\newtheorem{definition}{Definition}

\title{On the Hybrid Minimum Principle}

\author{Ali Pakniyat, \textit{Member, IEEE}, and Peter E. Caines, \textit{Life Fellow, IEEE}

\thanks{This work is supported by the Natural Sciences and Engineering Research Council of Canada (NSERC) and the Automotive Partnership Canada (APC).}
\thanks{Ali Pakniyat is with the Department of Mechanical Engineering, University of Michigan, Ann Arbor, USA and Peter E. Caines is with the Centre for Intelligent Machines (CIM) and the Department of Electrical and Computer Engineering, McGill University, Montreal, QC, Canada \newline
        {\tt\small pakniyat@umich.edu, peterc@cim.mcgill.ca}}%
}

\begin{document}

\bstctlcite{IEEEexample:BSTcontrol}

\maketitle
\thispagestyle{empty}
\pagestyle{empty}



\begin{abstract}

The Hybrid Minimum Principle (HMP) is established for the optimal control of deterministic hybrid systems with both autonomous and controlled switchings and jumps where state jumps at the switching instants are permitted to be accompanied by changes in the dimension of the state space. 
First order variational analysis is performed via the needle variation methodology and the necessary optimality conditions are established in the form of the HMP. A feature of special interest in this work is the explicit presentations of boundary conditions on the Hamiltonians and the adjoint processes before and after switchings and jumps. In addition to an analytic example, the HMP results are illustrated for the optimal control of an electric vehicle with transmission, where the modelling of the powertrain requires the consideration of both autonomous and controlled switchings accompanied by dimension changes.

\end{abstract}




\section{Introduction}

\IEEEPARstart{T}{he} Minimum Principle (MP), also called the Maximum Principle in the pioneering work of Pontryagin et al. \cite{Pontryagin}, is a milestone of systems and control theory that led to the emergence of optimal control as a distinct field of research. This principle states that any optimal control along with the optimal state trajectory must solve a two-point boundary value problem in the form of an extended Hamiltonian canonical system, as well as satisfying an extremization condition of the Hamiltonian function. Whether the extreme value is maximum or minimum depends on the sign convention used for the Hamiltonian definition. 

The main objective of this paper is the presentation and proof of the Minimum Principle for hybrid systems, i.e. the generalization of the MP for control systems with both continuous and discrete states and dynamics. It should be remarked that due to the development of hybrid systems theory in different scientific communities  which are motivated by various applications, the domains of definition of hybrid systems do not necessarily intersect in a general class of systems. For instance, in computer science hybrid systems are viewed as finite automata interacting with an analogue environment, and therefore the emphasis is often on the discrete event dynamics
\cite{PuriVaraiya, AlurHenzingerLafferrierePappas, AlurDangIvancic, ClarkeFehnkerHanKroghOuaknineStursbergTheobald, TiwariKhanna, Broucke, HelwaCaines, CoronaGiuaSeatzu}, while in the Control Systems community, the continuous dynamics is more dominant in the discussion. Even in hybrid systems stability theory (see e.g. \cite{GoebelSanfeliceTeel, LiberzonSwitching, LiberzonHespanhaMorse, Hespanha, BranickyTAC, DecarloBranickyPetterssonLennartson, JohanssonRantzer, VanDerSchaftSchumacher}) the considered structures for hybrid control inputs are different from the admissible set of input values considered for optimal control purposes. Moreover, the definitions and the underlying assumptions for the class of hybrid optimal control problems in Hybrid Dynamic Programming (HDP) \cite{BensoussanMenaldi, DharmattiRamaswamy, BarlesDharmattiRamaswamy, BranickyBorkerMitter, ShaikhPECVerification, PECEgerstedtMalhame, ScholligPECEgerstedt, DaSilvaDeSousaPereira, HedlundRantzer} differ from those of the \mbox{Hybrid Minimum Principle (HMP) literature}. 

The formulation of the HMP by Clarke and Vinter \cite{ClarkeVinterMultiprocess, ClarkeVinterOptimal}, referred to by them as ``optimal multiprocesses'', provides a Minimum Principle for hybrid systems of a very general nature in which switching conditions are regarded as constraints in the form of set inclusions and the dynamics of the constituent processes are governed by (possibly nonsmooth) differential inclusions. 
A similar philosophy is followed by Sussmann \cite{SussmannNonSmooth, Sussmann} where a nonsmooth MP is presented for hybrid systems possessing a general class of switching structures. Due to the generality of the considered structures in \cite{ClarkeVinterMultiprocess, ClarkeVinterOptimal, SussmannNonSmooth, Sussmann} degeneracy is not precluded, therefore additional hypotheses (typically of a controllabilty nature) need to be imposed to make the HMP results significantly informative (see e.g. \cite{CainesClarkeLiuVinter} for more discussion). 

An alternative philosophy, followed by Shaikh and Caines \cite{ShaikhPEC}, Garavello and Piccoli \cite{GaravelloPiccoli}, Taringoo and Caines \cite{FarzinPECSIAM}, and Pakniyat and Caines \cite{APPEC2017TAC} is to ensure the validity of the HMP in a non-degenerate form by introducing hypotheses on the dynamics, transitions and switching events. Then by performing first order variational analysis via the needle variation methodology, the necessary optimality conditions are established in the form of the HMP, with the emphasis of theoretical developments on generalization of the class of hybrid systems and on relaxation of regularity assumptions (see e.g. \cite{JafarpourLewisTopology} for a discussion on regulaty requirements in control theory). Moreover, non-degeneracy provided by this approach is advantageous in the development of numerical algorithms (see e.g. \cite{ShaikhPECOptimalityZone, TaringooPEC2011, AxelssonWardiEgerstedtVerriest, BoccadoroWardiEgerstedtVerriest, GonzalezVasudevanKamgarpourSastryBajcsyTomlin, ZhaoMohanVasudevan, ZhuAntsaklis, PassenbergLeiboldStursberg, Cowlagi, MamakoukasMacIverMurphey}). Other, prior, versions of the HMP which appeared in its development within hybrid system theory are to be found in the work of Riedinger and Kratz \cite{RiedingerKratz}, Xu and Antsaklis \cite{XuAntsaklis}, Azhmyakov, Boltyanski and Poznyak \cite{AzhmyakovBoltyanskiPoznyak}, and Dmitruk and Kaganovich \cite{DmitrukKaganovich, DmitrukKaganovich2011, DmitrukKaganovich2011Quadratic}.

In past work of the authors (see \cite{APPEC2017TAC, APPECCDC2014, APPECCDC2013}), a unified general framework for hybrid optimal control problems is presented within which the HMP, HDP, and their mutual relationship are valid. Distinctive aspects in this work are the presence of state dependent switching costs, the consideration of both autonomous and controlled switchings and jumps, and the possibility of state space and control space dimension changes. The latter aspect is of particular importance for systems with hybrid dynamics induced by restrictions of certain degrees of freedom (e.g. single and double support modes in legged locomotion \cite{GrizzleBook} and fixed gear modes and transitioning phases in automotive systems \cite{APPEC2017NAHS, APPEC2ADHS2015} presented in Section \ref{sec:EVexample}). Within this general framework, it is proved that along optimal trajectories of a hybrid system, the adjoint process in the HMP, and the gradient of the value function in HDP are equal almost everywhere (see \cite{APPECCDC2014} for a proof method based on variations over optimal trajectories, and \cite{APPEC2017TAC} for variations over general (i.e. not necessarily optimal) trajectories). Illustrative analytic examples are provided in \cite{APPECIFAC2014, APPECCDC2015, APPEC1ADHS2015}. 

The organisation of the paper is as follows: In \mbox{Section \ref{sec:HybridSystems}} a definition of hybrid systems is presented that covers a general class of nonlinear systems on Euclidean spaces with autonomous and controlled switchings and jumps allowed at the switching states and times. Section \ref{sec:HOCP} presents a general class of hybrid optimal control problems with a large range of running, terminal and switching costs. The regularity assumptions in Sections  \ref{sec:HybridSystems} and \ref{sec:HOCP} are attempted to be minimal and they are imposed primarily to ensure the existence and uniqueness of solutions as well as continuous dependence on initial conditions. Further generalizations such as the lying of the system's vector fields in Riemannian spaces \cite{FarzinPECSIAM,FarzinPEC}, nonsmooth assumptions \cite{Sussmann, SussmannNonSmooth, ClarkeVinterOptimal, ClarkeVinterMultiprocess, BensoussanMenaldi, DharmattiRamaswamy}, state-dependence of the control value sets \cite{GaravelloPiccoli}, and \mbox{stochastic hybrid} systems \cite{APPECCDC2016}, as well as restrictions to certain subclasses, such as those with regional dynamics \cite{PECEgerstedtMalhame, ScholligPECEgerstedt}, and with specified families of jumps \cite{BensoussanMenaldi, BranickyBorkerMitter, DharmattiRamaswamy, BarlesDharmattiRamaswamy}, become possible through variations and extensions of the framework presented here. 

The main result which is the statement and the proof of the Hybrid Minimum Principle is presented in Section \ref{sec:HMP}. Distinctive aspects of this work are the explicit presentation of the boundary conditions on the Hamiltonians and adjoint processes (in contrast to their implicit expressions in \cite{ClarkeVinterMultiprocess, ClarkeVinterOptimal, SussmannNonSmooth, Sussmann, GaravelloPiccoli}), the relaxation of the regularity requirements (relative to e.g. \cite{ShaikhPEC, FarzinPECSIAM}) and the presence of both autonomous and controlled switchings and jumps with switching costs and the possibility of state space dimension change (where only subsets of these features have been considered for the presentation of other versions of the HMP). Discussion on generalization of the HMP for time-varying vector fields, costs, and switching manifolds is provided in Section \ref{sec:HMPTV}. 

To illustrate the results, an analytic example is provided in Section \ref{sec:AnalyticExample} and in Section \ref{sec:EVexample} the hybrid model of an electric vehicle equipped with a dual-stage planetary transmission (presented in \cite{MechMachThry2015, USPatent2017}) is studied (see also \cite{APPEC2017NAHS, APPEC1ADHS2015}).
This example highlights some of the key features of the hybrid systems framework presented in this work.
In particular, the modelling of the powertrain requires the consideration of both autonomous and controlled state jumps, some of which are accompanied by changes in the dimension of the state space due to the changing degree of freedom during the transition period.  Furthermore, the corresponding hybrid automaton diagram for the full system (presented in Figure \ref{fig:HybridAutomata}) exhibits a majority of the permitted behaviour of the completely general automaton in the definition of hybrid systems in Section \ref{sec:HybridSystems} and \ref{sec:HOCP}. Moreover, there is a genuine restriction by the automaton imposed on discrete transitions expressed in (i.e. corresponding to) those in the state transition structure displayed in Figure \ref{fig:HybridAutomata}.

\section{Hybrid Systems}
\label{sec:HybridSystems}
\begin{definition}
A (deterministic) \textit{hybrid system (structure)} $\mathbb{H}$ is a septuple
\begin{equation}
\mathbb{H}=\left\{ H,I,\Gamma,A,F,\Xi,\mathcal{M}\right\} \label{Hybrid System}
\end{equation}
where the symbols in the expression and their governing assumptions are defined as below.
\newline

\noindent\textbf{A0:} $H:=\coprod_{q\in Q}\mathbb{R}^{n_{q}}$ is called the \textit{(hybrid) state space}
of the hybrid system $\mathbb{H}$, where $\coprod$ denotes disjoint union, i.e. $\coprod_{q\in Q}\mathbb{R}^{n_{q}} = \bigcup_{q\in Q}\big\{ (q,x) : x \in \mathbb{R}^{n_{q}} \big\}$, where

$Q=\left\{ 1,2,...,\left|Q\right|\right\} \equiv\left\{ q_{1},q_{2},...,q_{\left|Q\right|}\right\} ,\left|Q\right|<\infty$, is a finite set of \textit{discrete states (components)}, and

$\left\{ \mathbb{R}^{n_{q}}\right\} _{q\in Q}$ is a family of finite dimensional continuous valued state spaces, where $n_{q}\leq n<\infty$ for all $q\in Q$.

$I:=\Sigma\times U$ is the set of system input values, where

$\Sigma$ with $\left|\Sigma\right|<\infty$ is the set of discrete state transition and continuous state jump events extended with the identity element,

$U=\left\{ U_{q}\right\} _{q\in Q}$ is the set of \textit{admissible input control values}, where each $U_q \subset \mathbb{R}^{m_q}$ is a compact set in $\mathbb{R}^{m_q}$.

The set of admissible (continuous) control inputs $\mathcal{U}\left(U\right):=L_{\infty}\left(\left[t_{0},T_{*}\right),U\right)$, is defined to be the set of all measurable functions that are bounded up to a set of measure zero on $\left[t_{0},T_{*}\right),T_{*}<\infty$. The boundedness property necessarily holds since admissible inputs take values in the compact set $U$.

$\Gamma:H\times\Sigma\rightarrow H$ is a time independent (partially defined) \textit{discrete state transition map}.

$\Xi:H\times\Sigma\rightarrow H$ is a time independent (partially defined) \textit{continuous state jump transition map}. For all $\xi \in\Xi$, the functions $\xi_{\sigma} \equiv \xi(\cdot , \sigma) : \mathbb{R}^{n_q} \rightarrow  \mathbb{R}^{n_p}$, $p\in A\left(q,\sigma\right)$ are assumed to be continuously differentiable in the
continuous state $x \in \mathbb{R}^{n_q}$. 

$A:Q\times\Sigma\rightarrow Q$ denotes both a deterministic finite automaton and the automaton's associated transition function on the state space $Q$ and event set $\Sigma$, such that for a discrete state $q\in Q$ only the discrete controlled and uncontrolled transitions into the $q$-dependent subset $\left\{ A\left(q,\sigma\right),\sigma\in\Sigma\right\} \subset Q$ occur under the projection of $\Gamma$ on its $Q$ components: $\Gamma:Q\times\mathbb{R}^{n}\times\Sigma\rightarrow H|_{Q}$. In other words, $\Gamma$ can only make a discrete state transition in a hybrid state $\left(q,x\right)$ if the automaton $A$ can make the corresponding transition in $q$.

$F$ is an indexed collection of \textit{vector fields} $\left\{ f_{q}\right\} _{q\in Q}$ such that there exist $k_{f_q}\geq1$ for which $f_{q}\in C^{k_{f_q}} \left(\mathbb{R}^{n_q}\times U_q\rightarrow\mathbb{R}^{n_q}\right)$ satisfies a joint uniform Lipschitz condition, i.e., there exists $L_{f}<\infty$ such that $\left\Vert f_{q}\left(x_{1},u_{1}\right)-f_{q}\left(x_{2},u_{2}\right)\right\Vert \leq L_{f}\left(\left\Vert x_{1}-x_{2}\right\Vert + \left\Vert u_{1}-u_{2}\right\Vert \right)$ for all $x, x_{1},x_{2}\in\mathbb{R}^{n_q}$, $u, u_1, u_2\in U_q$, $q\in Q$. 

$\mathcal{M}=\left\{ m_{\alpha}:\alpha\in Q\times Q,\right\}$ denotes a collection of \textit{switching manifolds} such that, for any ordered pair $\alpha \equiv \left(\alpha_1,\alpha_2\right) =\left(q,r\right)$, $m_{\alpha}$ is a smooth, i.e. $C^{\infty}$ codimension $k$ sub-manifold of $\mathbb{R}^{n_q}$, $k\in\left\{1, \cdots , n_q \right\} $, described locally by $m_{\alpha}=\left\{ x: m_{\alpha}\left(x\right)=0 \right\} $, and possibly with boundary $\partial m_{\alpha}$. It is assumed that $m_{\alpha}\cap m_{\beta}=\emptyset$, whenever $\alpha_1 = \beta_1$ but $\alpha_2 \neq \beta_2$, for all $\alpha,\beta\in Q\times Q$.
\hfill $\square$
\end{definition}
We note that the case where $m_{\alpha}$ is identified with its reverse ordered version $m_{\bar{\alpha}}$ giving $m_{\alpha}=m_{\bar{\alpha}}$, is not ruled out by this definition, even in the non-trivial case $m_{p,p}$ where $\alpha_1 = \alpha_2 = p$. The former case corresponds to the common situation where the switching of vector fields at the passage of the continuous trajectory in one direction through a switching manifold is reversed if a reverse passage is performed by the continuous trajectory, while the latter case corresponds to the standard example of the bouncing ball. 


Switching manifolds will function in such a way that whenever a trajectory governed by the controlled vector field meets the switching manifold transversally there is an autonomous switching to another controlled vector field or there is a jump transition in the continuous state component, or both. A transversal arrival on a switching manifold $m_{q,r}$, at state $x_q \in m_{q,r}=\left\{ x \in \mathbb{R}^{n_{q}} : m_{q,r}\left(x\right)=0 \right\}$ occurs whenever
\begin{equation}
{\nabla m_{q,r} \left(x_q\right)}^T f_q \left(x_q,u_q\right) \neq 0 ,
\label{TransversalityOfTrajectoriesToManifolds}
\end{equation}
for $u_q\in U_q$, and $q,r \in Q$.  It is assumed that:


\noindent\textbf{A1:} The initial state $h_{0}:=\left(q_{0},x\left(t_{0}\right)\right)\in H$ is such that $m_{q_{0},q_{j}}\left(x_{0}\right)\neq0$, for all $q_{j}\in Q$. 
\hfill $\square$

\begin{definition}
\label{def:HybridInputProcess}
A \textit{hybrid input process} is a pair $I_{L}\equiv I_{L}^{\left[t_{0},t_{f}\right)}:= \left(S_L,u\right)$ defined on a half open interval $\left[t_{0},t_{f}\right)$, $t_{f}<\infty$, where $u\in{\cal U}$ and $S_L = \big(\left(t_{0},\sigma_{0}\right), \left(t_{1},\sigma_{1}\right), \cdots,$ $\left(t_{L},\sigma_{L}\right) \big)$, $L<\infty$, is a finite \textit{hybrid sequence of switching events} consisting of a strictly increasing sequence of times $\tau_L := \left\{ t_{0},t_{1},t_{2},\ldots,t_{L}\right\}$ and a \textit{discrete event sequence} $\sigma$ with $\sigma_0 = id$ and $\sigma_i \in \Sigma$, $i \in \left\{1,2,\cdots,L\right\}$.
\hfill $\square$
\end{definition}

\begin{definition}
\label{def:HybridStateProcess}
A \textit{hybrid state process} (or \textit{trajectory}) is a triple $\left(\tau_L,q,x\right)$ consisting of the sequence of switching times $\tau_L = \left\{ t_{0},t_{1},\ldots,t_{L}\right\}$, $L<\infty$, the associated sequence of discrete states $q=\left\{ q_{0},q_{1},\ldots,q_{L}\right\}$, and the sequence $x\left(\cdot\right)=\left\{ x_{q_{0}}\left(\cdot\right),x_{q_{1}}\left(\cdot\right),\ldots,x_{q_{L}}\left(\cdot\right)\right\}$ of piece-wise differentiable functions $x_{q_{i}}\left(\cdot\right):\left[t_{i},t_{i+1}\right)\rightarrow\mathbb{R}^{n}$.\hfill $\square$
\end{definition}

\begin{definition}
\label{def:InputStateTrajectory}
The \textit{input-state trajectory} for the hybrid system $\mathbb{H}$ satisfying A0 and A1 is a hybrid input $I_L = \left(S_L,u\right)$ together with its corresponding hybrid state trajectory $\left(\tau_L,q,x\right)$ defined over $\left[t_{0},t_{f}\right),t_{f}<\infty$, such that it satisfies:

\begin{enumerate}[(i)]
\item \textit{Continuous State Dynamics:} The continuous state component $x\left(\cdot\right) =\big\{ x_{q_{0}}\left(\cdot\right), $ $ x_{q_{1}}\left(\cdot\right),\ldots,x_{q_{L}}\left(\cdot\right)\big\}$ is a piecewise continuous function which is almost everywhere differentiable and on each time segment specified by $\tau_L$ satisfies the dynamics equation
\begin{equation}
{\dot{x}_{q_{i}}\left(t\right) = f_{q_{i}}\left(x_{q_{i}}\left(t\right),u\left(t\right)\right)}, \hspace{1 cm} {a.e.\; t\in\left[t_{i},t_{i+1}\right)},
\end{equation}
with the initial conditions
\begin{align}
x_{q_{0}}\left(t_{0}\right) &=x_{0}
\\
x_{q_{i}}\left(t_{i}\right) &= \xi_{\sigma_{i}}\left(x_{q_{i-1}}\left(t_{i}-\right)\right) := \xi_{\sigma_{i}}\left(\lim_{t\uparrow t_{i}}x_{q_{i-1}}\left(t\right)\right)
\end{align}
for $\left(t_{i},\sigma_{i}\right) \in S_L$. In other words, $x\left(\cdot\right)=\left\{ x_{q_{0}}\left(\cdot\right),x_{q_{1}}\left(\cdot\right),\ldots,x_{q_{L}}\left(\cdot\right)\right\}$ is a piecewise continuous function which is almost everywhere differentiable and is such that each $x_{q_{i}}\left(\cdot\right)$ satisfies
\begin{equation}
x_{q_{i}}\left(t\right)=x_{q_{i}}\left(t_{i}\right)+\int_{t_{i}}^{t}f_{q_{i}}\left(x_{q_{i}}\left(s\right),u\left(s\right)\right)ds
\end{equation}
for $t \in \left[t_{i},t_{i+1}\right)$.

\item \textit{Autonomous Discrete Transition Dynamics:} An autonomous (uncontrolled) discrete state transition from $q_{i-1}$ to $q_i$ together with a continuous state jump $\xi_{\sigma_i}$ occurs at the \textit{autonomous switching time} $t_i$ if $x_{q_{i-1}}\left(t_{i}-\right):=\lim_{t\uparrow t_{i}}x_{q_{i-1}}\left(t\right)$ satisfies a switching manifold condition of the form
\begin{equation}
m_{q_{i-1}q_i}\left(x_{q_{i-1}}\left(t_{i}-\right)\right) = 0
\end{equation}
for $q_i\in Q$, where $m_{q_{i-1}q_i}\left(x\right) = 0$ defines a $\left(q_{i-1},q_i\right)$ switching manifold and it is not the case that either $\left(i\right)$ $x\left(t_{i}-\right)\in \partial m_{q_{i-1}q_i}$ or $\left(ii\right)$ $f_{q_{i-1}}\left(x\left(t_{i}-\right),u\left(t_{i}-\right)\right) \perp \nabla m_{q_{i-1}q_i}\left(x\left(t_{i}-\right)\right)$, i.e. $t_i$ is not a manifold termination instant (see \cite{PECHybridNotes}). With the assumptions A0 and A1 in force, such a transition is well defined and labels the event $\sigma_{q_{i-1}q_i} \in \Sigma$, that corresponds to the hybrid state transition
\begin{multline}
h\left(t_{i}\right) \equiv \left(q_{i},x_{q_{i}}\left(t_{i}\right)\right) 
\\
= \left(\Gamma  \left(q_{i-1}, x_{q_{i-1}}\left(t_{i}-\right) ,\sigma_{q_{i-1}q_i}\right),\xi_{\sigma_{q_{i-1}q_i}}  \left(x_{q_{i-1}}\left(t_{i}-\right)\right)\right)
\end{multline}

\item \textit{Controlled Discrete Transition Dynamics:} A controlled discrete state transition together with a controlled continuous state jump $\xi_{\sigma_i}$ occurs at the \textit{controlled discrete event time} $t_i$ if $t_i$ is not an autonomous discrete event time and if there exists a controlled discrete input event $\sigma_{q_{i-1}q_i} \in \Sigma$ for which
\begin{multline}
h\left(t_{i}\right) \equiv \left(q_{i},x_{q_{i}}\left(t_{i}\right)\right) 
\\
= \left(\Gamma  \left(q_{i-1}, x_{q_{i-1}}\left(t_{i}-\right) ,\sigma_{q_{i-1}q_i}\right),\xi_{\sigma_{q_{i-1}q_i}}  \left(x_{q_{i-1}}\left(t_{i}-\right)\right)\right)
\end{multline}
with $\left(t_{i},\sigma_{q_{i-1}q_{i}}\right) \in S_L$ and $q_{i}\in A\left(q_{i-1}\right)$.
\hfill $\square$
\end{enumerate}
\end{definition}

\textbf{A2:} For a specified sequence of discrete states $\left\{q_i\right\}_{i=0}^{L}$, the class of input-state trajectories is non-empty. In other words, there exist $S_L = \big(\left(t_{0},\sigma_{0}\right), \left(t_{1},\sigma_{1}\right), \cdots, \left(t_{L},\sigma_{L}\right) \big) \equiv \big(\left(t_{0},q_{0}\right), \left(t_{1},q_{1}\right), \cdots, \left(t_{L},q_{L}\right) \big)$ and $u_{q_i} \in L_{\infty}\left(\left[t_{i},t_{i+1}\right), U_{q_i}\right)$ that together with its corresponding hybrid state process form an input-state trajectory in Definition \ref{def:InputStateTrajectory}.
\hfill $\square$

\begin{theorem}
\label{theorem:ExistenceUniqueness}
{{\cite{PECHybridNotes}}} 
A hybrid system $\mathbb{H}$ with an initial hybrid state $\left(q_{0},x_{0}\right)$ satisfying assumptions A0 and A1 possesses a unique hybrid input-state trajectory on $\left[t_{0},T_{**}\right)$, where $T_{**}$ is the least of
\begin{enumerate}[(i)]
\item $T_{*} \leq \infty $, where $\left[t_{0},T_{*}\right)$ is the temporal domain of the definition of the hybrid system,
\item a manifold termination instant $T_{*}$ of the trajectory $h\left(t\right) = h\left(t,\left(q_0,x_0\right),\left(S_L,u\right) \right)$, $t\geq t_0$, at which either $x\left(T_{*}-\right) \in \partial m_{q\left(T_*-\right)q\left(T_*\right)}$ or $f_{q\left(T_*-\right)}\left(x\left(T_*-\right),u\left(T_*-\right)\right) \perp \nabla m_{q\left(T_*-\right)q\left(T_*\right)}\left(x\left(T_{*}-\right)\right)$.
\hfill $\square$
\end{enumerate}
\end{theorem} 
We note that Zeno times, i.e. accumulation points of discrete transition times, are ruled out by A2. 

\begin{lemma}{{\cite{APPEC2017TAC}}} 
State processes of a hybrid system satisfying Assumptions A0-A2 are continuously dependent on their initial conditions. In other words, for a given $\left\{q_i\right\}_{i=0}^{L}$ and an initial continuous state $x_0 \in \mathbb{R}^{n_{q_0}}$, there exist a neighbourhood $N\left(x_0\right)$ and a constant $0<K<\infty$ such that
\begin{equation}
\left\Vert x\left(t_{f};s,x_{s}\right) - x\left(t_{f};t_0,x_{0}\right) \right\Vert \leq K\left(\left\Vert x_{s}-x_{0}\right\Vert ^{2}+\left|s-t_0\right|^{2}\right)^{\frac{1}{2}},
\end{equation} for $s \geq t_0$ and $x_s \in N\left(x_0\right)$.
\hfill $\square$
\end{lemma}



\section{Hybrid Optimal Control Problems}
\label{sec:HOCP}
\noindent\newline
\textbf{A3:} Let $\left\{ l_{q}\right\} _{q\in Q},l_{q}\in C^{n_{l}}\left(\mathbb{R}^{n}\times U\rightarrow\mathbb{R}_{+}\right),n_{l}\geq1$,
be a family of cost functions with $n_l = 2$ unless otherwise stated; $\left\{ c_{\sigma}\right\} _{\sigma\in\Sigma}\in C^{n_{c}}\left(\mathbb{R}^{n}\times\Sigma\rightarrow\mathbb{R}_{+}\right),n_{c}\geq1$,
be a family of switching cost functions; and $g\in C^{n_{g}}\left(\mathbb{R}^{n}\rightarrow\mathbb{R}_{+}\right),n_{g}\geq1$,
be a terminal cost function satisfying the following assumptions:

\begin{enumerate}[(i)]
\item There exists $K_{l}<\infty$ and $1\leq\gamma_{l}<\infty$ such that
$\left|l_{q}\left(x,u\right)\right|\leq K_{l}\left(1+\left\Vert x\right\Vert ^{\gamma_{l}}\right)$ and $\left|l_{q}\left(x_{1},u_{1}\right)-l_{q}\left(x_{2},u_{2}\right)\right|\leq K_{l}\left(\left\Vert x_{1}-x_{2}\right\Vert +\left\Vert u_{1}-u_{2}\right\Vert \right)$, for all
$x\in\mathbb{R}^{n},u\in U,q\in Q$.

\item There exists $K_{c}<\infty$ and $1\leq\gamma_{c}<\infty$ such that
$\left|c_{\sigma}\left(x\right)\right|\leq K_{c}\left(1+\left\Vert x\right\Vert ^{\gamma_{c}}\right)$,
$x\in\mathbb{R}^{n},\sigma\in\Sigma$.

\item There exists $K_{g}<\infty$ and $1\leq\gamma_{g}<\infty$ such that
$\left|g\left(x\right)\right|\leq K_{g}\left(1+\left\Vert x\right\Vert ^{\gamma_{g}}\right)$,
$x\in\mathbb{R}^{n}$.
\hfill $\square$
\end{enumerate}

Consider the initial time $t_{0}$, final time $t_{f}<\infty$, and initial
hybrid state $h_{0}=\left(q_{0},x_{0}\right)$. With the number
of switchings $L$ held fixed, the set of all hybrid input trajectories  in Definition \ref{def:HybridInputProcess} with exactly $L$ switchings is denoted by $\bm{I_{L}}$, and for all $I_{L} :=\left(S_{L},u\right) \in \bm{I_{L}}$ the hybrid switching sequences take the form
$S_{L}= \left\{ \left(t_{0},id\right),\left(t_{1},\sigma_{q_{0}q_{1}}\right),  \ldots, \left(t_{L},\sigma_{q_{L-1}q_{L}}\right)\right\}
\equiv\left\{ \left(t_{0},q_{0}\right),\left(t_{1},q_{1}\right),\ldots,\left(t_{L},q_{L}\right)\right\}$ and the corresponding continuous control inputs are of the form $u\in\mathcal{U} = \bigcup_{i=0}^{L} L_{\infty}\left(\left[t_i,t_{i+1}\right),U\right)$, where $t_{L+1}=t_f$.

Let $I_{L}$ be a hybrid input trajectory that by Theorem \ref{theorem:ExistenceUniqueness} results in a unique hybrid state process. Then hybrid performance functions for the corresponding hybrid input-state trajectory are defined as
\begin{multline}
J\left(t_{0},t_{f},h_{0},L;I_{L}\right):=
\sum_{i=0}^{L}\int_{t_{i}}^{t_{i+1}}l_{q_{i}}\left(x_{q_{i}}\left(s\right),u\left(s\right)\right)ds
\\
 +\sum_{j=1}^{L}c_{\sigma_{q_{j-1}q_{j}}}\left(t_{j},x_{q_{j-1}}\left(t_{j}-\right)\right)+g\left(x_{q_{L}}\left(t_{f}\right)\right)
\label{Hybrid Cost}
\end{multline}

\subsection{Bolza Hybrid Optimal Control Problem}
\label{subsec:BHOCP}
\begin{definition}
\label{def:BHOCP}
The {Bolza Hybrid Optimal Control Problem} (BHOCP) is defined as the infimization of the hybrid cost \eqref{Hybrid Cost} over the family of hybrid input trajectories $\bm{I_{L}}$, i.e.
\begin{equation}
J^{o}\left(t_{0},t_{f},h_{0},L\right)=\inf_{I_{L} \in \bm{I_{L}}}J\left(t_{0},t_{f},h_{0},L;I_{L}\right)\label{HOCP}
\end{equation}
\hfill $\square$
\end{definition}

\vspace{-7pt}

\subsection{Mayer Hybrid Optimal Control Problem}
\label{subsec:MHOCP}
\begin{definition}
\label{def:MHOCP}
The {Mayer Hybrid Optimal Control Problem} (MHOCP) is defined as a special case of the BHOCP where $l_q\left(x,u\right) = 0$ for all $q \in Q$, and $c_{\sigma}\left(x\left(t_{j}-\right)\right)$ $ = 0$ for all $\sigma \in \Sigma$.
\hfill $\square$
\end{definition}

\subsection{Relationship between Bolza and Mayer Hybrid Optimal Control Problems}
\label{subsec:BHOCPandMHOCPrelationship}
In general, a BHOCP can be converted into an MHOCP with the introduction of the auxiliary state component $z$ and the extension of the continuous valued state to
\begin{equation}
\hat{x}_{q}:=\left[\begin{array}{c}
z_{q}\\
x_{q}
\end{array}\right]\label{ExtendedState}
\end{equation}

With the definition of the augmented vector fields as
\begin{equation}
\dot{\hat{x}}_{q}=\hat{f}_{q}\left(\hat{x},u\right):=\left[\begin{array}{c}
l_{q}\left(x,u\right)\\
f_{q}\left(x,u\right)
\end{array}\right],\label{ExtendedField}
\end{equation}
subject to the initial condition
\begin{equation}
\hat{h}_{0}=\left(q_{0},\hat{x}_{q_{0}}\left(t_{0}\right)\right)=\left(q_{0},\left[\begin{array}{c}
0\\
x_{0}
\end{array}\right]\right),\label{ExtendedIC}
\end{equation}
and with the switching boundary conditions governed by the extended jump
function defined as
\begin{equation}
\hat{x}\left(t_{j}\right)=\hat{\xi}\left(\hat{x}\left(t_{j}-\right)\right):=\left[\begin{array}{c}
z\left(t_{j}-\right)+c\left(x\left(t_{j}-\right)\right)\\
\xi\left(x\left(t_{j}-\right)\right)
\end{array}\right],\label{ExtendedJump2}
\end{equation}
the cost  \eqref{Hybrid Cost} of the BHOCP turns into the Mayer form with
\begin{equation}
J\left(t_{0},t_{f},\hat{h}_{0},L;I_{L}\right):=\hat{g}\left(\hat{x}_{q_{L}}\left(t_{f}\right)\right),\label{Mayer - full transformation-1}
\end{equation}
where 
\begin{equation}
\hat{g}\left(\hat{x}_{q_{L}}\left(t_{f}\right)\right)=z\left(t_{f}\right)+g\left(x\left(t_{f}\right)\right)\,.
\label{gHatDefinition}
\end{equation}

\hfill $\square$

\section{The Hybrid Minimum Principle (HMP)}
\label{sec:HMP}

\begin{theorem}
\label{theorem:HMP}
Consider the hybrid system $\mathbb{H}$ subject to assumptions A0-A3, and the HOCP \eqref{HOCP} for the hybrid performance function \eqref{Hybrid Cost}.
Define the family of system Hamiltonians by 
\begin{equation}
H_{q}\left(x_q,\lambda_q,u_q\right)=\lambda_q^{T}f_{q}\left(x_q,u_q\right)+l_{q}\left(x_q,u_q\right),\label{Hamiltonian - Bolza}
\end{equation}
$x_q, \lambda_q \in \mathbb{R}^{n_{q}}$, $u_q \in U_{q}$, $q\in Q$, and let $\left\{q_i\right\}_{i=0}^{L}$ be a specified sequence of discrete states with its associated set of switchings. Then for an optimal input $u^{o}$ and along the corresponding optimal trajectory $x^{o}$, there exists an adjoint process $\lambda^{o}$ such that
\begin{equation}
H_{q}\left(x^{o}_{q},\lambda^{o}_{q},u^{o}_{q}\right)\leq H_{q}\left(x^{o}_{q},\lambda^{o}_{q},v\right),\label{HminWRTu}
\end{equation}
for all $v\in U_{q}$, 
where $(x^{o},\lambda^{o})$ satisfy
\begin{align}
\dot{x}^{o}_{q} &=\frac{\partial H_{q}}{\partial \lambda_q}\left(x_q,\lambda_q,u_q\right), \label{StateDynamics}
\\
\dot{\lambda}^{o}_{q} &=-\frac{\partial H_{q}}{\partial x_q}\left(x_q,\lambda_q,u_q\right),\label{lambda dynamics}
\end{align}
almost everywhere $\; t\in\left[t_{0},t_{f}\right]$, subject to
\begin{align}
x^o_{q_0}\left(t_0\right) &=x_0, \label{StateIC}
\\
x^o_{q_j}\left(t_{j}\right) &=\xi_{\sigma_j}\left(x^o_{q_{j-1}}\left(t_{j}-\right)\right), \label{StateBC} \\
\lambda^{o}_{q_L}\left(t_{f}\right) &=\nabla g\left(x^{o}_{q_L}\left(t_{f}\right)\right),\label{LambdaTerminal}\\
\lambda^{o}_{q_{j-1}}\left(t_{j}-\right)\equiv\lambda^{o}_{q_{j-1}}\left(t_{j}\right) &={\nabla\xi_{\sigma_j}}^{T}\lambda^{o}_{q_j}\left(t_{j}+\right) + \nabla c_{\sigma_j} +p\nabla m_{q_{j-1}q_j},\label{LambdaBC}
\end{align}
where $p\in\mathbb{R}$ when $t_{j}$ indicates the time of an autonomous switching, subject to the switching manifold condition $m_{q_{j-1}q_j}\big(x^o_{q_{j-1}}(t_{j}-)\big) = 0$, and $p=0$ when $t_{j}$ indicates the time of a controlled switching. Moreover, at both autonomous and controlled switching instants $t_{j}$, the Hamiltonian satisfies
\begin{samepage}
\begin{multline}
H_{q_{j-1}} \left. \left(x^{o},\lambda^{o},u^{o}\right)\right|_{t_{j}-} \equiv H_{q_{j-1}}  \left(t_{j}\right)
\\
=H_{q_{j}} \left(t_{j}\right) \equiv H_{q_{j}}\left.  \left(x^{o},\lambda^{o},u^{o}\right)\right|_{t_{j}+} .\label{Hamiltonian jump}
\end{multline}
\hfill $\square$
\end{samepage}
\end{theorem}

\begin{proof}
First, in part 
A, we study a needle variation to the optimal input at the last location $u^o_{q_L}$ at a Lebesgue instant\footnote{See e.g. \cite{AgrachevSachkov} for the definition of Lebesgue points. For any ${u \in L_{\infty}([t_i, t_{i+1} ], U)}$, $u$ may be modified on a set of measure zero so that all points are Lebesgue points (see e.g. \cite{Rudin}).} $t \in \left(t_L,t_{L+1}\right] \equiv \left(t_L,t_f\right]$  to derive the Hamiltonian canonical equations \eqref{StateDynamics} and \eqref{lambda dynamics}, the adjoint terminal condition \eqref{LambdaTerminal}, and the Hamiltonian minimization condition \eqref{HminWRTu} in that location. This part of the proof is similar to the proof of the classical Pontryagin Minimum Principle.

Next, in part 
B, we perform a variation in the penultimate, $L-1^{\text{st}}$, location in order to obtain $\left(i\right)$ Hamiltonian canonical equations \eqref{StateDynamics} and \eqref{lambda dynamics}, and $\left(ii\right)$ the Hamiltonian minimization condition \eqref{HminWRTu} at the location $q_{L-1}$, as well as $\left(iii\right)$ the boundary conditions \eqref{StateBC} and \eqref{LambdaBC}, and $\left(iv\right)$ the Hamiltonian boundary condition \eqref{Hamiltonian jump} at time $t_L$.

Then, in part 
C, we extend the analysis for a general switching instant $t_j$ and prove that $\left(i\right)$ to $\left(iv\right)$ above hold for all locations.


In order to provide the simplest derivation of the main result we employ the Mayer version of the problem throughout the analysis. The equivalence of the Mayer and the Bolza formulations given in Section \ref{subsec:BHOCPandMHOCPrelationship} then yields the Hybrid Minimum Principle in both forms.

Due to space limitation, the arguments of a function are sometimes written as superscripts, or the arguments are not displayed in full whenever their identification is clear, e.g. $f_{q_j}\left(x_{q_j}(s),u_{q_j}(s)\right) \equiv f_{q_j}^{\left(x_{q_j}^{(s)},u_{q_j}^{(s)}\right)} \equiv f_{q_j}^{\left(s\right)}$, etc. Other notation conventions are defined upon their first appearance.

\subsection{The last discrete state location}
\label{subsec:Last}
First, consider a Lebesgue time $t \in \left(t_L,t_{L+1}\right] \equiv \left(t_L,t_f\right]$ and the evolution of the optimal state $\hat{x}^{o}\left(\tau\right)$, $\tau \in \left[t_0,t_f\right]$, governed by the set of differential equations
\begin{equation}
\frac{d}{d \tau}\hat{x}_{q_i}^o = \hat{f}_{q_{i}}\left(\hat{x}_{q_{i}}^{o}\left(\tau\right),u_{q_{i}}^{o}\left(\tau\right)\right), \hspace{10pt} \tau \in \left[t_i,t_{i+1}\right).
\label{ODEforXhat}
\end{equation}

We perform a needle variation at a Lebesgue time $t$ in the form of
\begin{equation}
u^{\epsilon}\left(\tau\right)=\begin{cases}
\begin{array}{lccclcc}
u_{q_{j-1}}^{o}\left(\tau\right) &  & \text{if} &  & \tau\in\left[t_{j-1},t_{j}\right) &  & 1\leq j\leq L\\
u_{q_{L}}^{o}\left(\tau\right) &  & \text{if} &  & \tau\in\left[t_{L},t-\epsilon\right)\\
v &  & \text{if} &  & \tau\in\left[t-\epsilon,t\right)\\
u_{q_{L}}^{o}\left(\tau\right) &  & \text{if} &  & \tau\in\left[t,t_{f}\right]
\end{array}\end{cases}.\label{NeedleVariationL}
\end{equation}

This corresponds to a perturbed trajectory $\hat{x}^{\epsilon}\left( \tau \right), \tau \in \left[t_0,t_f\right]$. Denoting $\delta\hat{x}_{q_{L}}^{\epsilon}\left(\tau\right):=\hat{x}_{q_{L}}^{\epsilon}\left(\tau\right)-\hat{x}_{q_{L}}^{o}\left(\tau\right)$, it necessarily satisfies $\delta\hat{x}_{q_{i}}^{\epsilon}\left(\tau\right) = 0$
for $\tau \in \left[t_0,t\right)$, $0 \leq i \leq L$, and for $\tau \in \left[t,t_f\right]$ it satisfies
\begin{multline}
\delta\hat{x}_{q_{L}}^{\epsilon}\left(\tau\right)=\int_{t-\epsilon}^{t}\left[\hat{f}{}_{q_{L}}\left(\hat{x}_{q_{L}}^{\epsilon}\left(s\right),v\right)-\hat{f}{}_{q_{L}}\left(\hat{x}_{q_{L}}^{o}\left(s\right),u_{q_{L}}^{o}\left(s\right)\right)\right]ds
\\
+\int_{t}^{\tau}\left[\hat{f}{}_{q_{L}}\left(\hat{x}_{q_{L}}^{\epsilon}\left(s\right),u_{q_{L}}^{o}\left(s\right)\right)-\hat{f}{}_{q_{L}}\left(\hat{x}_{q_{L}}^{o}\left(s\right),u_{q_{L}}^{o}\left(s\right)\right)\right]ds,
\end{multline}

Defining the first order state variation as
\begin{equation}
y\left(\tau\right):=\left.\frac{d}{d\epsilon}\hat{x}^{\epsilon}\left(\tau\right)\right|_{\epsilon=0}\equiv\underset{\epsilon\rightarrow0}{\lim}\frac{1}{\epsilon}\delta\hat{x}^{\epsilon}\left(\tau\right),
\end{equation}
the dynamics and boundary conditions of the first order state sensitivity are derived
\begin{align}
&\frac{d}{d\tau}y_{q_{L}}\left(\tau\right)=\frac{\partial f_{q_{L}}}{\partial x_{q_{L}}}\left(x_{q_{L}}^{o}\left(\tau\right),u_{q_{L}}^{o}\left(\tau\right)\right)y_{q_{L}}\left(\tau\right),
\label{yLdynamics}
\\
&y_{q_{L}}\left(t\right)=f_{q_{L}}\left(x_{q_{L}}^{o}\left(t\right),v\right)-f_{q_{L}}\left(x_{q_{L}}^{o}\left(t\right),u_{q_{L}}^{o}\left(t\right)\right).
\label{yLtNoSwitch}
\end{align}

Denoting the state transition matrix corresponding to \eqref{yLdynamics} by $\Phi_{q_{L}}$, it is shown by Linearization Theory (see e.g. \cite{PECHybridNotes, Sontag}) that
\begin{equation}
y_{q_{L}}\left(t_{f}\right)=\Phi_{q_{L}}\left(t_{f},t\right)\left[\hat{f}{}_{q_{L}}\left(\hat{x}_{q_{L}}^{o}\left(t\right),v\right)-\hat{f}{}_{q_{L}}\left(\hat{x}_{q_{L}}^{o}\left(t\right),u_{q_{L}}^{o}\left(t\right)\right)\right].
\label{yfNoSwitching}
\end{equation}

The optimality of $\hat{x}^o$ implies that
\begin{equation}
\hat{g}\left(\hat{x}_{q_{L}}^{\epsilon}\left(t_{f}\right)\right)\geq\hat{g}\left(\hat{x}_{q_{L}}^{o}\left(t_{f}\right)\right),
\end{equation}
which is equivalent to
\begin{equation}
\left.\frac{d}{d\epsilon}J\left(u^{\epsilon}\right)\right|_{\epsilon=0}=\left[\frac{\partial\hat{g}}{\partial\hat{x}_{q_{L}}}\left(\hat{x}_{q_{L}}^{o}\left(t_{f}\right)\right)\right]^{T}y_{q_{L}}\left(t_{f}\right) \geq 0.
\label{qLoptimalityTf}
\end{equation}

Substitution of \eqref{yfNoSwitching} into \eqref{qLoptimalityTf} results in
\begin{multline}
\frac{\partial\hat{g}}{\partial\hat{x}_{q_{L}}}\left(\hat{x}_{q_{L}}^{o}\left(t_{f}\right)\right)^{T}\Phi_{q_{L}}\left(t_{f},t\right)\hat{f}_{q_{L}}\left(\hat{x}_{q_{L}}^{o}\left(t\right),v\right)
\\
\geq\frac{\partial\hat{g}}{\partial\hat{x}_{q_{L}}}\left(\hat{x}_{q_{L}}^{o}\left(t_{f}\right)\right)^{T}\Phi_{q_{L}}\left(t_{f},t\right)\hat{f}_{q_{1}}\left(\hat{x}_{q_{L}}^{o}\left(t\right),u_{q_{L}}^{o}\left(t\right)\right).
\label{InequalityFromOptimality}
\end{multline}

Setting 
\begin{equation}
{\left.{\hat{\lambda}_{q_{L}}^{o}}\right.}^{T}\left(t\right)\equiv\left[\lambda_{0,q_{L}}^{o}\left(t\right),{\lambda_{q_{L}}^{o}}^{T}\left(t\right)\right]=\frac{\partial\hat{g}}{\partial\hat{x}_{q_{L}}}\left(\hat{x}_{q_{L}}^{o}\left(t_{f}\right)\right)^{T}\Phi_{q_{L}}\left(t_{f},t\right),
\label{LambdaHatDefinition}
\end{equation}
for $t\in \left(t_L,t_f\right]$ and evaluating it at $t=t_f$ we obtain
\begin{equation}
\hat{\lambda}_{q_{L}}^{o}\left(t_{f}\right)=\frac{\partial\hat{g}}{\partial\hat{x}_{q_{L}}}\left(\hat{x}_{q_{L}}^{o}\left(t_{f}\right)\right),
\end{equation}
where, by the definition \eqref{gHatDefinition} for $\hat{g}$, this is equivalent to
\begin{align}
\lambda_{0,q_{L}}^{o}\left(t_{f}\right)&=1,\label{Lambda0Terminal}
\\
\lambda_{q_{L}}^{o}\left(t_{f}\right)&=\frac{\partial g}{\partial x_{q_{L}}}\left(x_{q_{L}}^{o}\left(t_{f}\right)\right)\equiv\nabla g\left(x_{q_{L}}^{o}\left(t_{f}\right)\right).
\end{align}

Also by differentiation of \eqref{LambdaHatDefinition} with respect to $t$ we obtain
\begin{multline}
\frac{d}{dt}\hat{\lambda}_{q_{L}}^{o}\left(t\right)=-\frac{\partial\hat{f}_{q_{L}}}{\partial\hat{x}_{q_{L}}}^{T}\left[\Phi_{q_{L}}\left(t_{f},t\right)\right]^{T}\frac{\partial\hat{g}}{\partial\hat{x}_{q_{L}}}\left(\hat{x}_{q_{L}}^{o}\left(t_{f}\right)\right)
\\
=-\frac{\partial\hat{f}_{q_{L}}}{\partial\hat{x}_{q_{L}}}^{T}\hat{\lambda}_{q_{L}}^{o}\left(t\right),
\end{multline}
which is equivalent to
\begin{align}
\frac{d}{dt}\lambda_{0,q_{L}}^{o}=\; & 0, \label{Lambda0Dynamics}
\\
\begin{split}
\frac{d}{dt}\lambda_{q_{L}}^{o}=&-\left(\frac{\partial l_{q_{L}}\left(x_{q_{L}}^{o}\left(t\right),u_{q_{L}}^{o}\left(t\right)\right)}{\partial x_{q_{L}}}\right)\lambda_{0,q_{L}}^{o}\left(t\right)
\\
&-\left(\frac{\partial f_{q_{L}}\left(x_{q_{L}}^{o}\left(t\right),u_{q_{L}}^{o}\left(t\right)\right)}{\partial x_{q_{L}}}\right)^{T}\lambda_{q_{L}}^{o}\left(t\right). \label{LambdaDynamics}
\end{split}
\end{align}

The zero dynamics \eqref{Lambda0Dynamics} with the terminal condition \eqref{Lambda0Terminal} gives $\lambda_{0,q_{L}}^{o}\left(t\right) = 1$, for all $t\in \left(t_L,t_f\right)$, and equation \eqref{LambdaDynamics} is equivalent to
\begin{equation}
\dot{\lambda}_{q_{L}}^{o}=-\frac{\partial H_{q_{L}}\left(x_{q_{L}}^{o},\lambda_{q_{L}}^{o},u_{q_{L}}^{o}\right)}{\partial x_{q_{L}}},
\end{equation}
which is valid on $\left(t_L,t_f\right)$ and where by definition
\begin{equation}
H_{q_{L}}\left(x_{q_{L}},\lambda_{q_{L}},u_{q_{L}}\right)=l_{q_{L}}\left(x_{q_{L}},u_{q_{L}}\right)+\lambda_{q_{L}}^{T}f_{q_{L}}\left(x_{q_{L}},u_{q_{L}}\right).
\label{qLHamiltonianDefinition}
\end{equation}

From the definition of Hamiltonian \eqref{qLHamiltonianDefinition} and through a simple differentiation, the Hamiltonian canonical equation \eqref{StateDynamics} for the state is also verified.

Also from \eqref{InequalityFromOptimality} and \eqref{qLHamiltonianDefinition} the Hamiltonian minimization 
\begin{equation}
H_{q_{L}}\left(x_{q_{L}}^{o},\lambda_{q_{L}}^{o},u_{q_{L}}^{o}\right)\leq H_{q_{L}}\left(x_{q_{L}}^{o},\lambda_{q_{L}}^{o},v\right) ,
\end{equation}
is obtained for all $v\in U_{q_{L}}$.

\subsection{The penultimate location}
\label{subsec:Penultimate}
Now consider a needle variation at time $t \in \left(t_{L-1},t_L\right]$ in the form of
\begin{equation}
u^{\epsilon}\left(\tau\right)=\left\{ \begin{array}{lclc}
u_{q_{j-1}}^{o}\left(\tau\right), &  & \tau\in\left[t_{j-1},t_{j}\right), & 1\leq j\leq L-1,\\
u_{q_{L-1}}^{o}\left(\tau\right), &  & \tau\in\left[t_{L},t-\epsilon\right),\\
v, &  & \tau\in\left[t-\epsilon,t\right),\\
u_{q_{L-1}}^{o}\left(\tau\right), &  & \tau\in\left[t,t_{L}-\delta^{\epsilon}\right),\\
u_{q_{L}}^{o}\left(t_{L}\right), &  & \tau\in\left[t_{L}-\delta^{\epsilon},t_{L}\right),\\
u_{q_{L}}^{o}\left(\tau\right), &  & \tau\in\left[t_{L},t_{f}\right],
\end{array}\right.,
\label{NeedleVariationL-1}
\end{equation}
where $\delta^{\epsilon} \geq 0$ corresponds to the case when the perturbed trajectory arrives on the switching manifold $\hat{m} \left(\hat{x}\right) := m_{q_{L-1}q_L}\left(x\right) = 0$ at an earlier instant. The case with a later arrival time, i.e. $\delta^{\epsilon} \leq 0$ is handled in a similar fashion, and the case of a controlled switching, i.e. with no switching manifold, can be derived similarly by setting $\delta^{\epsilon} = 0$.

For $\tau \in \left[t,t_L-\delta^\epsilon\right)$ we may write
\begin{multline}
\delta\hat{x}_{q_{L-1}}^{\epsilon}\left(\tau\right):=\hat{x}_{q_{L-1}}^{\epsilon}\left(\tau\right)-\hat{x}_{q_{L-1}}^{o}\left(\tau\right)
\\
=\int_{t-\epsilon}^{t}\left[\hat{f}{}_{q_{L-1}}\left(\hat{x}_{q_{L-1}}^{\epsilon}\left(s\right),v\right)-\hat{f}{}_{q_{L-1}}\left(\hat{x}_{q_{L-1}}^{o}\left(s\right),u_{q_{L-1}}^{o}\left(s\right)\right)\right]ds
\\
+\int_{t}^{\tau}\Big[\hat{f}{}_{q_{L-1}}\left(\hat{x}_{q_{L-1}}^{\epsilon}\left(s\right),u_{q_{L-1}}^{o}\left(s\right)\right)
\\
-\hat{f}{}_{q_{L-1}}\left(\hat{x}_{q_{L-1}}^{o}\left(s\right),u_{q_{L-1}}^{o}\left(s\right)\right)\Big]ds,
\end{multline}

At $t_L$ the state of the optimal trajectory is determined by
\begin{multline}
\hat{x}_{q_{L}}^{o}\left(t_{L}\right)=\hat{\xi}\left(\hat{x}_{q_{L-1}}^{o}\left(t_{L}-\right)\right)
\\
=\hat{\xi}\left(\hat{x}_{q_{L-1}}^{o}\left(t_{L}-\delta^{\epsilon}\right)+\int_{t_{L}-\delta^{\epsilon}}^{t_{L}}\hat{f}_{q_{L-1}}\left(\hat{x}_{q_{L-1}}^{o}\left(\tau\right),u_{q_{L-1}}^{o}\left(\tau\right)\right)d\tau\right),
\end{multline}
and the state of the perturbed trajectory is calculated as
\begin{multline}
\hat{x}_{q_{L}}^{\epsilon}\left(t_{L}\right)=\hat{\xi}\left(\hat{x}_{q_{L-1}}^{\epsilon}\left(t_{L}-\delta^{\epsilon}-\right)\right)
\\
+\int_{t_{L}-\delta^{\epsilon}}^{t_{L}}\hat{f}_{q_{L}}\left(\hat{x}_{q_{L}}^{\epsilon}\left(\tau\right),u_{q_{L}}^{o}\left(t_{L}\right)\right)d\tau .
\end{multline}

Thus
\begin{multline}
\delta\hat{x}_{q_{L}}^{\epsilon}\left(t_{L}\right)=\hat{x}_{q_{L}}^{\epsilon}\left(t_{L}\right)-\hat{x}_{q_{L}}^{o}\left(t_{L}\right)
\\
=\hat{\xi}\left(\hat{x}_{q_{L-1}}^{\epsilon}\left(t_{L}-\delta^{\epsilon}-\right)\right)+\int_{t_{L}-\delta^{\epsilon}}^{t_{L}}\hat{f}_{q_{L}}\left(\hat{x}_{q_{L}}^{\epsilon}\left(\tau\right),u_{q_{L}}^{o}\left(t_{L}\right)\right)d\tau
\\
-\hat{\xi}\left(\hat{x}_{q_{L-1}}^{o}\left(t_{L}-\delta^{\epsilon}\right)+\int\limits_{t_{L}-\delta^{\epsilon}}^{t_{L}}\hat{f}_{q_{L-1}}\left(\hat{x}_{q_{L-1}}^{o}\left(\tau\right),u_{q_{L-1}}^{o}\left(\tau\right)\right)d\tau\right),
\end{multline}
and hence, with the definition of $\mu_L:=\underset{\epsilon\rightarrow0}{\lim}\frac{\delta^{\epsilon}}{\epsilon}$, the first order state sensitivity at $t_L$ is determined from
\begin{equation}
y_{q_{L}}\left(t_{L}\right)=\frac{\partial\hat{\xi}}{\partial\hat{x}_{q_{L-1}}}\left(\hat{x}_{q_{L-1}}^{o}\left(t_{L}-\right)\right)y_{q_{L-1}}\left(t_{L}-\right)+\mu_{L}\hat{f}_{q_{L},\hat{\xi}}^{\hat{\xi},q_{L-1}},
\label{yBeforeLastAfterTs}
\end{equation}
where \begin{multline}
\hat{f}_{q_L,\hat{\xi}}^{\hat{\xi},q_{L-1}} := \hat{f}_{q_{L}}\left(\hat{\xi}\left(\hat{x}_{q_{L-1}}^{o}\left(t_{L}-\right)\right),u_{q_{L}}^{o}\left(t_{L}\right)\right)
\\
 -\frac{\partial\hat{\xi}}{\partial\hat{x}_{q_{L-1}}}\left(\hat{x}_{q_{L-1}}^{o}\left(t_{L}-\right)\right)\hat{f}_{q_{L-1}}\left(\hat{x}_{q_{L-1}}^{o}\left(t_{L}-\right),u_{q_{L-1}}^{o}\left(t_{L}-\right)\right),
\end{multline}
and where $\mu_L = 0$ if $t_L$ is the time of a controlled switching since $\delta^{\epsilon} = 0$, and
\begin{equation}
\mu_L=\frac{\left[\frac{\partial\hat{m}\left(\hat{x}_{q_{L-1}}^{o}\left(t_{L}-\right)\right)}{\partial\hat{x}_{q_{L-1}}}\right]^{T}y_{q_{L-1}}\left(t_{L}-\right)}{\left[\frac{\partial\hat{m}\left(\hat{x}_{q_{L-1}}^{o}\left(t_{L}-\right)\right)}{\partial\hat{x}_{q_{L-1}}}\right]^{T}\hat{f}_{q_{L-1}}\left(\hat{x}_{q_{L-1}}^{o}\left(t_{L}-\right),u_{q_{L-1}}^{o}\left(t_{L}-\right)\right)}\,,
\label{MuL}
\end{equation}
in the case of an autonomous switching. In writing \eqref{MuL} we have employed the fact that 
\begin{equation}
\lim_{\epsilon\rightarrow0}\frac{1}{\epsilon}\left[\delta\hat{x}_{q_{L-1}}^{\epsilon}\left(t_{L}^{-}\right)-\int _{t_{L}-\delta^{\epsilon}}^{t_{L}}\hat{f}_{q_{L-1}}^{\left(\hat{x}^{o},\hat{u}^{o}\right)}d\tau\right]^{T}\frac{\partial\hat{m}\left(\hat{x}_{q_{L-1}}^{o}\left(t_{L}^{-}\right)\right)}{\partial\hat{x}_{q_{L-1}}}=0,
\end{equation}
because by the switching manifold conditions
$
\hat{m}\big(\hat{x}_{q_{L-1}}^{o}\left(t_{L}-\right)\big) = \hat{m}\big(\hat{x}_{q_{L-1}}^{\epsilon}\left(t_{L}-\delta^{\epsilon}-\right)\big) = 0,
$
and therefore
$
\lim_{\epsilon\rightarrow0}\frac{1}{\epsilon}\left[\hat{x}_{q_{L-1}}^{\epsilon}\left(t_{L}-\delta^{\epsilon}-\right)-\hat{x}_{q_{L-1}}^{o}\left(t_{L}-\right)\right]^T \frac{\partial\hat{m}\left(\hat{x}_{q_{L-1}}^{o}\left(t_{L}-\right)\right)}{\partial\hat{x}_{q_{L-1}}} = 0.
$
\vspace{10pt}

Similar to part 
A, the dynamics and boundary conditions of the first order state sensitivity are derived as
\begin{align}
&y_{q_{L-1}}\left(t\right)=f_{q_{L-1}}\left(x_{q_{L-1}}^{o}\left(t\right),v\right)-f_{q_{L-1}}\left(x_{q_{L-1}}^{o}\left(t\right),u_{q_{L-1}}^{o}\left(t\right)\right),
\\
&\frac{d}{d\tau}y_{q_{L-1}}\left(\tau\right)=\frac{\partial f_{q_{L-1}}}{\partial x_{q_{L-1}}}\left(x_{q_{L-1}}^{o}\left(\tau\right),u_{q_{L-1}}^{o}\left(\tau\right)\right)y_{q_{L-1}}\left(\tau\right),
\\
&y_{q_{L}}\left(t_{L}\right)=\frac{\partial\hat{\xi}}{\partial\hat{x}_{q_{L-1}}}\left(\hat{x}_{q_{L-1}}^{o}\left(t_{L}-\right)\right)y_{q_{L-1}}\left(t_{L}-\right)+\mu_{L}\hat{f}_{q_{L},\hat{\xi}}^{\hat{\xi},q_{L-1}},
\\
&\frac{d}{d\tau}y_{q_{L}}\left(\tau\right)=\frac{\partial f_{q_{L}}}{\partial x_{q_{L}}}\left(x_{q_{L}}^{o}\left(\tau\right),u_{q_{L}}^{o}\left(\tau\right)\right)y_{q_{L}}\left(\tau\right)d\tau,
\end{align}
and thus
\begin{multline}
y_{q_{L}}\left(t_{f}\right)=\mu_{L}\Phi_{q_{L}}\left(t_{f},t_{L}\right)\hat{f}_{q_{L},\hat{\xi}}^{\hat{\xi},q_{L-1}}\\+\Phi_{q_{L}}^{\left(t_{f},t_{L}\right)}\frac{\partial\hat{\xi}}{\partial\hat{x}_{q_{L-1}}}\Phi_{q_{L-1}}^{\left(t_{L},t\right)}\left[\hat{f}{}_{q_{L-1}}^{\left(\hat{x}_{q_{L-1}}^{o\left(t\right)},v\right)}-\hat{f}{}_{q_{L-1}}^{\left(\hat{x}_{q_{L-1}}^{o\left(t\right)},u_{q_{L-1}}^{o\left(t\right)}\right)}\right].
\end{multline}

Therefore, the optimality condition \eqref{qLoptimalityTf} is expressed as
\begin{multline}
\left[\frac{\partial\hat{g}}{\partial\hat{x}_{q_{L}}}^{T}\Phi_{q_{L}}^{\left(t_{f},t_{L}\right)}\frac{\partial\hat{\xi}}{\partial\hat{x}_{q_{L-1}}}+p\left[\frac{\partial\hat{m}}{\partial\hat{x}_{q_{L-1}}}\right]^{T}\right]
\\
\hfill \Phi_{q_{L-1}}\left(t_{L},t\right)\hat{f}{}_{q_{L-1}}\left(\hat{x}_{q_{L-1}}^{\epsilon}\left(t\right),v\right)\\\geq\left[\frac{\partial\hat{g}}{\partial\hat{x}_{q_{L}}}^{T}\Phi_{q_{L}}^{\left(t_{f},t_{L}\right)}\frac{\partial\hat{\xi}}{\partial\hat{x}_{q_{L-1}}}+p\left[\frac{\partial\hat{m}}{\partial\hat{x}_{q_{L-1}}}\right]^{T}\right] \hfill
\\
\hfill\Phi_{q_{L-1}}\left(t_{L},t\right)\hat{f}{}_{q_{L-1}}\left(\hat{x}_{q_{L-1}}^{o}\left(t\right),u_{q_{L-1}}^{o}\left(t\right)\right),
\label{PreHamiltonianBeforeSwitching}
\end{multline}
where
\begin{equation}
p=\frac{\frac{\partial\hat{g}}{\partial\hat{x}_{q_{L}}}\left(\hat{x}_{q_{L}}^{o}\left(t_{f}\right)\right)^{T}\Phi_{q_{L}}\left(t_{f},t_{L}\right)\hat{f}_{q_{L},\hat{\xi}}^{\hat{\xi},q_{L-1}}}{\left[\frac{\partial\hat{m}\left(\hat{x}_{q_{L-1}}^{o}\left(t_{L}-\right)\right)}{\partial\hat{x}_{q_{L-1}}}\right]^{T}\hat{f}_{q_{L-1}}\left(\hat{x}_{q_{L-1}}^{o}\left(t_{L}-\right),u_{q_{L-1}}^{o}\left(t_{L}-\right)\right)}.
\label{pDefinition}
\end{equation}

Setting
\begin{multline}
{\left.\hat{\lambda}_{q_{L-1}}^{o}\right.}^{T}\left(t\right)
=\scalebox{1.5}{\Bigg[}\frac{\partial\hat{g}\left(\hat{x}_{q_{L}}^{o}\left(t_{f}\right)\right)}{\partial\hat{x}_{q_{L}}}^{T}\Phi_{q_{L}}\left(t_{f},t_{L}\right)\frac{\partial\hat{\xi}\left(\hat{x}_{q_{L-1}}^{o}\left(t_{L}-\right)\right)}{\partial\hat{x}_{q_{L-1}}}
\\
+p\left[\frac{\partial\hat{m}\left(\hat{x}_{q_{L-1}}^{o}\left(t_{L}-\right)\right)}{\partial\hat{x}_{q_{L-1}}}\right]^{T}\scalebox{1.5}{\Bigg]}\Phi_{q_{L-1}}\left(t_{L},t\right),
\label{LambdaHatDefinitionProof}
\end{multline}
for $t\in \left[t_{L-1},t_L\right]$ and evaluating it at $t=t_L$ we obtain
\begin{multline}
{\left.\hat{\lambda}_{q_{L-1}}^{o}\right.}^{T}\left(t_{L}\right) =\frac{\partial\hat{g}\left(\hat{x}_{q_{L}}^{o}\left(t_{f}\right)\right)}{\partial\hat{x}_{q_{L}}}^{T}\Phi_{q_{L}}\left(t_{f},t_{L}\right)\frac{\partial\hat{\xi}\left(\hat{x}_{q_{L-1}}^{o}\left(t_{L}-\right)\right)}{\partial\hat{x}_{q_{L-1}}}
\\
+p\left[\frac{\partial\hat{m}\left(\hat{x}_{q_{L-1}}^{o}\left(t_{L}-\right)\right)}{\partial\hat{x}_{q_{L-1}}}\right]^{T}
\\
={\left.\hat{\lambda}_{q_{L}}^{o}\right.}^{T}\left(t_{L}+\right)\frac{\partial\hat{\xi}\left(\hat{x}_{q_{L-1}}^{o}\left(t_{L}-\right)\right)}{\partial\hat{x}_{q_{L-1}}}+p\left[\frac{\partial\hat{m}\left(\hat{x}_{q_{L-1}}^{o}\left(t_{L}-\right)\right)}{\partial\hat{x}_{q_{L-1}}}\right]^{T}
\label{LambdaHatBC}
\end{multline}

By the definition of $\hat{\xi}$ in \eqref{ExtendedJump2}, we have
\begin{multline}
\frac{\partial\hat{\xi}\left(\hat{x}_{q_{L-1}}^{o}\left(t_{L}-\right)\right)}{\partial\hat{x}_{q_{L-1}}}=\left[\begin{array}{c}
\frac{\partial\hat{\xi}}{\partial z}\\
\frac{\partial\hat{\xi}}{\partial x}
\end{array}\right]
\\
=\left[\begin{array}{cccc}
\frac{\partial\left[z+c\right]}{\partial z} & \frac{\partial\left[z+c\right]}{\partial x_{1}} & \cdots & \frac{\partial\left[z+c\right]}{\partial x_{n}}\\
\frac{\partial\xi_{1}}{\partial z} & \frac{\partial\xi_{1}}{\partial x_{1}} & \cdots & \frac{\partial\xi_{1}}{\partial x_{n}}\\
\vdots & \vdots & \ddots & \vdots\\
\frac{\partial\xi_{n}}{\partial z} & \frac{\partial\xi_{n}}{\partial x_{1}} & \cdots & \frac{\partial\xi_{n}}{\partial x_{n}}
\end{array}\right] \hspace{14pt}
\\
=\left[\begin{array}{cccc}
1 & \frac{\partial c}{\partial x_{1}} & \cdots & \frac{\partial c}{\partial x_{n}}\\
0 & \frac{\partial\xi_{1}}{\partial x_{1}} & \cdots & \frac{\partial\xi_{1}}{\partial x_{n}}\\
\vdots & \vdots & \ddots & \vdots\\
0 & \frac{\partial\xi_{n}}{\partial x_{1}} & \cdots & \frac{\partial\xi_{n}}{\partial x_{n}}
\end{array}\right]=\left[\begin{array}{cc}
1 & \nabla c^{T}\\
\mathbf{0} & \nabla\xi
\end{array}\right],
\label{ExtendedJumpExpansion}
\end{multline}
and since also $\frac{\partial m}{\partial z} = 0$ we have
\begin{equation}
\frac{\partial\hat{m}\left(\hat{x}_{q_{L-1}}^{o}\left(t_{L}-\right)\right)}{\partial\hat{x}_{q_{L-1}}}=\left[\begin{array}{c}
\frac{\partial\hat{m}}{\partial z}\\
\frac{\partial\hat{m}}{\partial x}
\end{array}\right]=\left[\begin{array}{c}
0\\
{\nabla}{m}
\end{array}\right].
\end{equation}

Hence, \eqref{LambdaHatBC} is equivalent to
\begin{multline}
\hat{\lambda}_{q_{L-1}}^{o}\left(t_{L}\right)\equiv\left[\begin{array}{c}
\lambda_{q_{L-1},0}^{o}\left(t_{L}\right)\\
\lambda_{q_{L-1}}^{o}\left(t_{L}\right)
\end{array}\right]
\\
={\frac{\partial\hat{\xi}\left(\hat{x}_{q_{L-1}}^{o}\left(t_{L}-\right)\right)}{\partial\hat{x}_{q_{L-1}}}}^{T}\hat{\lambda}_{q_{L}}^{o}\left(t_{L}+\right)+p\frac{\partial\hat{m}\left(\hat{x}_{q_{L-1}}^{o}\left(t_{L}-\right)\right)}{\partial\hat{x}_{q_{L-1}}}
\\
=\left[\begin{array}{cccc}
1 & 0 & \cdots & 0\\
\frac{\partial c}{\partial x_{1}} & \frac{\partial\xi_{1}}{\partial x_{1}} & \cdots & \frac{\partial\xi_{n}}{\partial x_{1}}\\
\vdots & \vdots & \ddots & \vdots\\
\frac{\partial c}{\partial x_{n}} & \frac{\partial\xi_{1}}{\partial x_{n}} & \cdots & \frac{\partial\xi_{n}}{\partial x_{n}}
\end{array}\right]\left[\begin{array}{c}
\lambda_{q_{L},0}^{o}\left(t_{L}+\right)\\
\lambda_{q_{L}}^{o}\left(t_{L}+\right)
\end{array}\right]+p\left[\begin{array}{c}
0\\
\nabla m
\end{array}\right]
\\
=\left[\begin{array}{c}
1\\
\nabla\xi^{T}\lambda_{q_{L}}^{o}\left(t_{L}+\right)+\nabla c+p\nabla m
\end{array}\right],
\end{multline}
i.e.
\begin{align}
\lambda_{q_{L-1},0}^{o}\left(t_{L}\right)&=1,
\\
\lambda_{q_{L-1}}^{o}\left(t_{L}\right)&=\nabla\xi^{T}\lambda_{q_{L}}^{o}\left(t_{L}+\right)+\nabla c+p\nabla m\,.
\label{DerivedLambdaBCforL-1}
\end{align}

Differentiating \eqref{LambdaHatDefinition} with respect to $t$ leads to
\begin{equation}
\frac{d}{dt}\hat{\lambda}_{q_{L-1}}^{o}\left(t\right)=-\left(\frac{\partial\hat{f}_{q_{L-1}}}{\partial\hat{x}_{q_{L-1}}}\left(\hat{x}_{q_{L-1}}^{o}\left(t\right),u_{q_{L-1}}^{o}\left(t\right)\right)\right)^{T}\hat{\lambda}_{q_{L-1}}^{o}\left(t\right),
\end{equation}
which is equivalent to
\begin{align}
\frac{d}{dt}\lambda_{q_{L-1},0}^{o}\left(t\right) &=0, &
\\
\frac{d}{dt}\lambda_{q_{L-1}}^{o}\left(t\right) &=-\left(\frac{\partial l_{q_{L-1}}\left(x_{q_{L-1}}^{o}\left(t\right),u_{q_{L-1}}^{o}\left(t\right)\right)}{\partial x_{q_{L-1}}}\right)\lambda_{0}^{o}\left(t\right) &
\\
\nonumber &-\left(\frac{\partial f_{q_{L-1}}\left(x_{q_{L-1}}^{o}\left(t\right),u_{q_{L-1}}^{o}\left(t\right)\right)}{\partial x_{q_{L-1}}}\right)^{T}\lambda_{q_{L-1}}^{o}\left(t\right).&\label{LambdaDynamicsBeforeSwitch}
\end{align}

Therefore, $\lambda_{q_{L-1},0}^{o}\left(t\right)=1$ for $t\in \left(t_{L-1},t_L\right)$ is obtained as before and 
\begin{equation}
\dot{\lambda}_{q_{L-1}}^o = - \frac{\partial H_{q_{L-1}} \left(x_{q_{L-1}}^o,\lambda_{q_{L-1}}^o,u_{q_{L-1}}^o\right) }{\partial x_{q_{L-1}}},
\end{equation}
holds for $t\in \left(t_{L-1},t_L\right)$ with the Hamiltonian defined as
\begin{multline}
H_{q_{L-1}}\left(x_{q_{L-1}},\lambda_{q_{L-1}},u_{q_{L-1}}\right)
\\
=l_{q_{L-1}}\left(x_{q_{L-1}},u_{q_{L-1}}\right)+\lambda_{q_{L-1}}^{T}f_{q_{L-1}}\left(x_{q_{L-1}},u_{q_{L-1}}\right).
\end{multline}

Also from \eqref{PreHamiltonianBeforeSwitching} the minimization of the Hamiltonian is concluded as
\begin{equation}
H_{q_{L-1}}\left(x_{q_{L-1}}^{o},\lambda_{q_{L-1}}^{o},u_{q_{L-1}}^{o}\right)\leq H_{q_{L-1}}\left(x_{q_{L-1}}^{o},\lambda_{q_{L-1}}^{o},v\right),
\end{equation}
for all $v\in U_{q_{L-1}}$. 

Evaluating both $H_{q_{L-1}}$ and $H_{q_L}$ at $t_L$ gives
\begin{multline}
\hspace{-6pt} H_{q_{L-1}}\left(t_{L}-\right)=l_{q_{L-1}}\left(x_{q_{L-1}}^{o}\left(t_{L}-\right),u_{q_{L-1}}^{o}\left(t_{L}-\right)\right) 
\\
\hfill +\lambda_{q_{L-1}}^{o}\left(t_{L}-\right)^{T}f_{q_{L-1}}\left(x_{q_{L-1}}^{o}\left(t_{L}-\right),u_{q_{L-1}}^{o}\left(t_{L}-\right)\right)
\\
~ =\left[\hat{\lambda}_{q_{L-1}}^{o}\left(t_{L}-\right)\right]^{T}\hat{f}_{q_{L-1}}\left(\hat{x}_{q_{L-1}}\left(t_{L}-\right),u_{q_{L-1}}^{o}\left(t_{L}-\right)\right) \hfill
\\
\underset{\eqref{LambdaHatBC}}{=}\left[\frac{\partial\hat{\xi}\left(\hat{x}_{q_{L-1}}^{o}\left(t_{L}^{-}\right)\right)}{\partial\hat{x}_{q_{L-1}}}^{T}\hat{\lambda}_{q_{L}}^{o}\left(t_{L}^{+}\right)+p\frac{\partial\hat{m}\left(\hat{x}_{q_{L-1}}^{o}\left(t_{L}^{-}\right)\right)}{\partial\hat{x}_{q_{L-1}}}\right]^{T}\hat{f}_{q_{L-1}}\left(t_{L}^{-}\right)
\\
\underset{\eqref{pDefinition}}{=}\scalebox{1.2}{\Bigg[}\hat{\lambda}_{q_{L}}^{o}\left(t_{L}+\right)^{T}\frac{\partial\hat{\xi}\left(\hat{x}_{q_{L-1}}^{o}\left(t_{L}-\right)\right)}{\partial\hat{x}_{q_{L-1}}} \hfill
\\
\hfill +\frac{\frac{\partial\hat{g}}{\partial\hat{x}_{q_{L}}}^{T}\Phi_{q_{L}}\left(t_{f},t_{L}\right)\hat{f}_{q_{L},\hat{\xi}}^{\hat{\xi},q_{L-1}}}{\frac{\partial\hat{m}}{\partial\hat{x}_{q_{L-1}}}^{T}\hat{f}_{q_{L-1}}\left(\hat{x}_{q_{L-1}}^{o\left(t_{L}^{-}\right)},u_{q_{L-1}}^{o\left(t_{L}^{-}\right)}\right)}\left[\frac{\partial\hat{m}}{\partial\hat{x}_{q_{L-1}}}\right]^{T}\scalebox{1.2}{\Bigg]}\hat{f}_{q_{L-1}}\left(t_{L}^{-}\right)
\\
~=\frac{\partial\hat{g}}{\partial\hat{x}_{q_{L}}}^{T}\Phi_{q_{L}}^{\left(t_{f},t_{L}\right)}\frac{\partial\hat{\xi}}{\partial\hat{x}_{q_{L-1}}}\hat{f}_{q_{L-1}}^{\left(\hat{x}_{q_{L-1}}\left(t_{L}-\right),u_{q_{L-1}}^{o}\left(t_{L}-\right)\right)} \hfill
\\
\hfill +\frac{\frac{\partial\hat{g}}{\partial\hat{x}_{q_{L}}}^{T}\Phi_{q_{L}}\left(t_{f},t_{L}\right)\hat{f}_{q_{L},\hat{\xi}}^{\hat{\xi},q_{L-1}}}{\frac{\partial\hat{m}}{\partial\hat{x}_{q_{L-1}}}^{T}\hat{f}_{q_{L-1}}^{\left(t_{L}-\right)}}\left[\frac{\partial\hat{m}}{\partial\hat{x}_{q_{L-1}}}\right]^{T}\hat{f}_{q_{L-1}}^{\left(t_{L}-\right)}
\\
\underset{\eqref{yBeforeLastAfterTs}}{=} \frac{\partial\hat{g}}{\partial\hat{x}_{q_{L}}}^{T}\Phi_{q_{L}}^{\left(t_{f},t_{L}\right)}\frac{\partial\hat{\xi}}{\partial\hat{x}_{q_{L-1}}}\hat{f}_{q_{L-1}}^{\left(t_{L}-\right)} \hfill
\\
\hfill +\frac{\partial\hat{g}}{\partial\hat{x}_{q_{L}}}^{T}\Phi_{q_{L}}^{\left(t_{f},t_{L}\right)}\left[\hat{f}_{q_{L}}^{\left(t_{L}\right)}-\frac{\partial\hat{\xi}}{\partial\hat{x}_{q_{L-1}}}\hat{f}_{q_{L-1}}^{\left(t_{L}-\right)}\right]
\\
~=\frac{\partial\hat{g}}{\partial\hat{x}_{q_{L}}}^{T}\Phi_{q_{L}}\left(t_{f},t_{L}\right)\hat{f}_{q_{L}}\left(\hat{\xi}\left(\hat{x}_{q_{L-1}}^{o}\left(t_{L}-\right)\right),u_{q_{L}}^{o}\left(t_{L}\right)\right) \hfill
\\
\underset{\eqref{LambdaHatDefinition}}{=}\hat{\lambda}_{q_{L}}^{o}\left(t_{L}+\right)^{T}\hat{f}_{q_{L}}\left(\hat{x}_{q_{L}}^{o}\left(t_{L}\right),u_{q_{L}}^{o}\left(t_{L}\right)\right) \hfill
\\
=l_{q_{L}}\left(x_{q_{L}}^{o}\left(t_{L}\right),u_{q_{L}}^{o}\left(t_{L}\right)\right)+\lambda_{q_{L}}^{o}\left(t_{L}+\right)^{T}f_{q_{L}}\left(x_{q_{L}}^{o}\left(t_{L}\right),u_{q_{L}}^{o}\left(t_{L}\right)\right)
\\
=H_{q_{L}}\left(t_{L}+\right),
\end{multline}
which is equivalent to \eqref{Hamiltonian jump}.

\subsection{Other locations}
\label{subsec:Other}
We now consider a needle variation at a general Lebesgue time $t \in \left(t_{n-1},t_n\right)$ in the form of
\begin{equation}
u^{\epsilon}\left(\tau\right)=\left\{ \begin{array}{lcll}
u_{q_{j-1}}^{o}\left(\tau\right), &  & \tau\in\left[t_{j-1},t_{j}\right), & \hspace{-17pt} 1\leq j\leq n-1,\\
u_{q_{n-1}}^{o}\left(\tau\right), &  & \tau\in\left[t_{n-1},t-\epsilon\right),\\
v, &  & \tau\in\left[t-\epsilon,t\right),\\
u_{q_{n-1}}^{o}\left(\tau\right), &  & \tau\in\left[t,t_{n}-\delta_{n}^{\epsilon}\right),\\
u_{q_{n}}^{o}\left(t_{n}\right), &  & \tau\in\left[t_{n}-\delta_{n}^{\epsilon},t_{n}\right),\\
u_{q_{k}}^{o}\left(\tau\right), &  & \tau\in\left[t_{k},t_{k+1}-\delta_{k+1}^{\epsilon}\right), & n\leq k\leq L,\\
u_{q_{k+1}}^{o}\left(t_{k+1}\right), \hspace{-5pt} &  & \tau\in\left[t_{k+1}-\delta_{k+1}^{\epsilon},t_{k+1}\right), & n\leq k<L.
\end{array}\right.
\label{NeedleVariationN-1}
\end{equation}

As before,
\begin{equation}
y_{q_{n-1}}\left(t_{n}-\right)
=\Phi_{q_{n-1}}^{\left(t_{n},t\right)}\left[\hat{f}{}_{q_{n-1}}\left(\hat{x}_{q_{n-1}}^{\epsilon\left(t\right)},v\right)-\hat{f}{}_{q_{n-1}}\left(\hat{x}_{q_{n-1}}^{o\left(t\right)},u_{q_{n-1}}^{o\left(t\right)}\right)\right],
\end{equation}
and
\begin{multline}
y_{q_{n}}\left(t_{n}\right)
\\
={\Bigg[}\frac{\partial\hat{\xi}_{\sigma_{n}}}{\partial\hat{x}_{q_{n-1}}}
+\frac{1}{\left[\frac{\partial\hat{m}_{q_{n-1}q_{n}}}{\partial\hat{x}_{q_{n-1}}}\right]^{T}\hat{f}{}_{q_{n-1}}^{\left(t_{n}^{-}\right)}}\,\hat{f}_{q_{n},\hat{\xi}_{\sigma_{n}}}^{\hat{\xi}_{\sigma_{n}},q_{n-1}}\left[\frac{\partial\hat{m}_{q_{n-1}q_{n}}}{\partial\hat{x}_{q_{n-1}}}\right]^{T}{\Bigg]}y_{q_{n-1}}^{\left(t_{n}^{-}\right)}. \hspace{-8pt}
\end{multline}

Therefore,
\begin{multline}
y_{q_{\text{L}}}\left(t_{f}\right)
\\
=\prod_{k=L}^{n}\left[\Phi_{q_{k}}\left(t_{k+1},t_{k}\right)\frac{\partial\hat{\xi}_{\sigma_{k}}}{\partial\hat{x}_{q_{k-1}}}+\gamma_{k}\,\hat{f}_{q_{k},\hat{\xi}_{\sigma_{k}}}^{\hat{\xi}_{\sigma_{k}},q_{k-1}}\left[\frac{\partial\hat{m}_{q_{k-1}q_{k}}}{\partial\hat{x}_{q_{k-1}}}\right]^{T}\right]
\\
\Phi_{q_{n-1}}^{\left(t_{n},t\right)}\left[\hat{f}{}_{q_{n-1}}\left(\hat{x}_{q_{n-1}}^{\epsilon\left(t\right)},v\right)-\hat{f}{}_{q_{n-1}}\left(\hat{x}_{q_{n-1}}^{o\left(t\right)},u_{q_{n-1}}^{o\left(t\right)}\right)\right],
\end{multline}
where
\begin{multline}
\hat{f}_{q_{k},\hat{\xi}_{\sigma_{k}}}^{\hat{\xi}_{\sigma_{k}},q_{k-1}}:=\hat{f}_{q_{k}}\left(\hat{\xi}_{\sigma_{k}}\left(\hat{x}_{q_{k-1}}^{o}\left(t_{k}^{-}\right)\right),u_{q_{k}}^{o}\left(t_{k}\right)\right)
\\
-\frac{\partial\hat{\xi}_{\sigma_{k}}}{\partial\hat{x}_{q_{k-1}}}\left(\hat{x}_{q_{k-1}}^{o}\left(t_{k}^{-}\right)\right)\hat{f}_{q_{k-1}}\left(\hat{x}_{q_{k-1}}^{o}\left(t_{k}^{-}\right),u_{q_{k-1}}^{o}\left(t_{k}^{-}\right)\right),
\end{multline}
and
\begin{equation}
\gamma_{k}:=\begin{cases}
0, & \hfill \text{controlled switching,}
\\
\frac{1}{\left[\frac{\partial\hat{m}_{q_{k-1}q_{k}}}{\partial\hat{x}_{q_{k-1}}}\right]^{T}\hat{f}{}_{q_{k-1}}^{\left(t_{k}^{-}\right)}}, & \hfill \text{autonomous switching.}
\end{cases}
\end{equation}

The optimality condition \eqref{qLoptimalityTf} is expressed as
\begin{multline}
\left[\frac{\partial\hat{g}}{\partial\hat{x}_{q_{L}}}\right]^{T}\prod_{k=L}^{n}\left[\Phi_{q_{k}}^{\left(t_{k+1},t_{k}\right)}\frac{\partial\hat{\xi}_{\sigma_{k}}}{\partial\hat{x}_{q_{k-1}}}+\gamma_{k}\,\hat{f}_{q_{k},\hat{\xi}_{\sigma_{k}}}^{\hat{\xi}_{\sigma_{k}},q_{k-1}}\left[\frac{\partial\hat{m}_{q_{k-1}q_{k}}}{\partial\hat{x}_{q_{k-1}}}\right]^{T}\right]
\\
\Phi_{q_{n-1}}^{\left(t_{n},t\right)}\left[\hat{f}{}_{q_{n-1}}\left(\hat{x}_{q_{n-1}}^{\epsilon\left(t\right)},v\right)-\hat{f}{}_{q_{n-1}}\left(\hat{x}_{q_{n-1}}^{o\left(t\right)},u_{q_{n-1}}^{o\left(t\right)}\right)\right]\geq0.
\label{PreHamiltonianBeforeManySwitchings}
\end{multline}

Setting
\begin{multline}
{\left.\hat{\lambda}_{q_{n-1}}^{o}\right.}^{T}\left(t\right)
=\left[\frac{\partial\hat{g}}{\partial\hat{x}_{q_{L}}}\right]^{T}\prod_{k=L}^{n}\scalebox{1.2}{\Bigg[}\Phi_{q_{k}}^{\left(t_{k+1},t_{k}\right)}\frac{\partial\hat{\xi}_{\sigma_{k}}}{\partial\hat{x}_{q_{k-1}}}
\\
+\gamma_{k}\,\hat{f}_{q_{k},\hat{\xi}_{\sigma_{k}}}^{\hat{\xi}_{\sigma_{k}},q_{k-1}}\left[\frac{\partial\hat{m}_{q_{k-1}q_{k}}}{\partial\hat{x}_{q_{k-1}}}\right]^{T}\scalebox{1.2}{\Bigg]}\Phi_{q_{n-1}}^{\left(t_{n},t\right)},
\end{multline}
for $t\in \left[t_{n-1},t_n\right]$, which is equivalent to 
\begin{multline}
\hat{\lambda}_{q_{n-1}}^{o}\left(t\right)=\left[\Phi_{q_{n-1}}^{\left(t_{n},t\right)}\right]^{T}\prod_{k=n}^{L}\scalebox{1.2}{\Bigg[}\left[\frac{\partial\hat{\xi}_{\sigma_{k}}}{\partial\hat{x}_{q_{k-1}}}\right]^{T}\left[\Phi_{q_{k}}^{\left(t_{k+1},t_{k}\right)}\right]^{T}
\\
\hfill+\gamma_{k}\frac{\partial\hat{m}_{q_{k-1}q_{k}}}{\partial\hat{x}_{q_{k-1}}}\left[\hat{f}_{q_{k},\hat{\xi}_{\sigma_{k}}}^{\hat{\xi}_{\sigma_{k}},q_{k-1}}\right]^{T}\scalebox{1.2}{\Bigg]}\frac{\partial\hat{g}}{\partial\hat{x}_{q_{L}}}
\\
=\left[\Phi_{q_{n-1}}^{\left(t_{n},t\right)}\right]^{T}\scalebox{1.2}{\Bigg[}\left[\frac{\partial\hat{\xi}_{\sigma_{n}}}{\partial\hat{x}_{q_{n-1}}}\right]^{T}\left[\Phi_{q_{n}}\left(t_{n+1},t_{n}\right)\right]^{T} \hfill
\\
+\gamma_{n}\frac{\partial\hat{m}_{q_{n-1}q_{n}}}{\partial\hat{x}_{q_{n-1}}}\left[\hat{f}_{q_{n},\hat{\xi}_{\sigma_{n}}}^{\hat{\xi}_{\sigma_{n}},q_{n-1}}\right]^{T}\scalebox{1.2}{\Bigg]}\prod_{k=n+1}^{L}\scalebox{1.2}{\Bigg[}\left[\frac{\partial\hat{\xi}_{\sigma_{k}}}{\partial\hat{x}_{q_{k-1}}}\right]^{T}\left[\Phi_{q_{k}}^{\left(t_{k+1},t_{k}\right)}\right]^{T}
\\
+\gamma_{k}\frac{\partial\hat{m}_{q_{k-1}q_{k}}}{\partial\hat{x}_{q_{k-1}}}\left[\hat{f}_{q_{k},\hat{\xi}_{\sigma_{k}}}^{\hat{\xi}_{\sigma_{k}},q_{k-1}}\right]^{T}\scalebox{1.2}{\Bigg]}\frac{\partial\hat{g}}{\partial\hat{x}_{q_{L}}},\label{LambdaHatGeneralDefinition}
\end{multline}
we may evaluate \eqref{LambdaHatGeneralDefinition} at $t=t_n$ to obtain
\begin{multline}
\hat{\lambda}_{q_{n-1}}^{o\left(t_{n}\right)}=\left[\left[\frac{\partial\hat{\xi}_{\sigma_{n}}}{\partial\hat{x}_{q_{n-1}}}\right]^{T}\left[\Phi_{q_{n}}^{\left(t_{n+1},t_{n}\right)}\right]^{T}+\gamma_{n}\frac{\partial\hat{m}_{q_{n-1}q_{n}}}{\partial\hat{x}_{q_{n-1}}}\left[\hat{f}_{q_{n},\hat{\xi}_{\sigma_{n}}}^{\hat{\xi}_{\sigma_{n}},q_{n-1}}\right]^{T}\right]
\\
\hfill\prod_{k=n+1}^{L}\left[\left[\frac{\partial\hat{\xi}_{\sigma_{k}}}{\partial\hat{x}_{q_{k-1}}}\right]^{T}\left[\Phi_{q_{k}}^{\left(t_{k+1},t_{k}\right)}\right]^{T}+\gamma_{k}\frac{\partial\hat{m}_{q_{k-1}q_{k}}}{\partial\hat{x}_{q_{k-1}}}\left[\hat{f}_{q_{k},\hat{\xi}_{\sigma_{k}}}^{\hat{\xi}_{\sigma_{k}},q_{k-1}}\right]^{T}\right]\frac{\partial\hat{g}}{\partial\hat{x}_{q_{L}}},\hspace{-9pt}
\end{multline}
or
\begin{multline}
\hspace{-9pt}
\hat{\lambda}_{q_{n-1}}^{o\left(t_{n}\right)}=\left[\frac{\partial\hat{\xi}_{\sigma_{n}}}{\partial\hat{x}_{q_{n-1}}}\right]^{T}\left[\Phi_{q_{n}}^{\left(t_{n+1},t_{n}\right)}\right]^{T}\prod_{k=n+1}^{L}\scalebox{1.2}{\Bigg[}\left[\frac{\partial\hat{\xi}_{\sigma_{k}}}{\partial\hat{x}_{q_{k-1}}}\right]^{T}\left[\Phi_{q_{k}}^{\left(t_{k+1},t_{k}\right)}\right]^{T}
\\
+\gamma_{k}\frac{\partial\hat{m}_{q_{k-1}q_{k}}}{\partial\hat{x}_{q_{k-1}}}\left[\hat{f}_{q_{k},\hat{\xi}_{\sigma_{k}}}^{\hat{\xi}_{\sigma_{k}},q_{k-1}}\right]^{T}\scalebox{1.2}{\Bigg]}\frac{\partial\hat{g}}{\partial\hat{x}_{q_{L}}}+\gamma_{n}\frac{\partial\hat{m}_{q_{n-1}q_{n}}}{\partial\hat{x}_{q_{n-1}}}\left[\hat{f}_{q_{n},\hat{\xi}_{\sigma_{n}}}^{\hat{\xi}_{\sigma_{n}},q_{n-1}}\right]^{T} \hfill
\\
\prod_{k=n+1}^{L}\left[\left[\frac{\partial\hat{\xi}_{\sigma_{k}}}{\partial\hat{x}_{q_{k-1}}}\right]^{T}\left[\Phi_{q_{k}}^{\left(t_{k+1},t_{k}\right)}\right]^{T}+\gamma_{k}\frac{\partial\hat{m}_{q_{k-1}q_{k}}}{\partial\hat{x}_{q_{k-1}}}\left[\hat{f}_{q_{k},\hat{\xi}_{\sigma_{k}}}^{\hat{\xi}_{\sigma_{k}},q_{k-1}}\right]^{T}\right]\frac{\partial\hat{g}}{\partial\hat{x}_{q_{L}}}. \hspace{-9pt}
\label{LambdaHatGeneralBC}
\end{multline}

Having established \eqref{LambdaHatBC}, we take the (backward) induction hypothesis as
\begin{multline}
\hat{\lambda}_{q_{n}}^{o}\left(\tau\right)=\Big[\Phi_{q_{n}}\left(t_{n+1},\tau\right)\Big]^{T}\prod_{k=n+1}^{L}\scalebox{1.2}{\Bigg[}\left[\frac{\partial\hat{\xi}_{\sigma_{k}}}{\partial\hat{x}_{q_{k-1}}}\right]^{T}\Big[\Phi_{q_{k}}\left(t_{k+1},t_{k}\right)\Big]^{T}
\\
+\gamma_{k}\frac{\partial\hat{m}_{q_{k-1}q_{k}}}{\partial\hat{x}_{q_{k-1}}}\left[\hat{f}_{q_{k},\hat{\xi}_{\sigma_{k}}}^{\hat{\xi}_{\sigma_{k}},q_{k-1}}\right]^{T}\scalebox{1.2}{\Bigg]}\frac{\partial\hat{g}}{\partial\hat{x}_{q_{L}}},
\label{LambdaInductionHypothesis}
\end{multline}
and denote the scalar product
\begin{multline}
p_{n}:=\gamma_{n}\left[\hat{f}_{q_{n},\hat{\xi}_{\sigma_{n}}}^{\hat{\xi}_{\sigma_{n}},q_{n-1}}\right]^{T}\prod_{k=n+1}^{L}\scalebox{1.2}{\Bigg[}\left[\frac{\partial\hat{\xi}_{\sigma_{k}}}{\partial\hat{x}_{q_{k-1}}}\right]^{T}\left[\Phi_{q_{k}}\left(t_{k+1},t_{k}\right)\right]^{T}
\\
+\gamma_{k}\frac{\partial\hat{m}_{q_{k-1}q_{k}}}{\partial\hat{x}_{q_{k-1}}}\left[\hat{f}_{q_{k},\hat{\xi}_{\sigma_{k}}}^{\hat{\xi}_{\sigma_{k}},q_{k-1}}\right]^{T}\scalebox{1.2}{\Bigg]}\frac{\partial\hat{g}}{\partial\hat{x}_{q_{L}}}.
\end{multline}

Then equation \eqref{LambdaHatGeneralBC} becomes
\begin{equation}
\hat{\lambda}_{q_{n-1}}^{o}\left(t_{n}\right)= \left[\frac{\partial\hat{\xi}_{\sigma_{n}}}{\partial\hat{x}_{q_{n-1}}}\right]^{T} \hat{\lambda}_{q_{n}}^{o}\left(t_{n}+\right)+p_{n}\frac{\partial\hat{m}_{q_{n-1}q_{n}}}{\partial\hat{x}_{q_{n-1}}}.
\label{LambdaInductiveStep}
\end{equation}

Since the induction hypothesis \eqref{LambdaInductionHypothesis} is proved to hold as \eqref{LambdaHatBC} for $n=L-1$, and since \eqref{LambdaInductionHypothesis} for $n$ implies \eqref{LambdaInductiveStep}, the boundary condition \eqref{LambdaBC} is deduced from \eqref{LambdaInductiveStep} in a similar way as shown in \eqref{ExtendedJumpExpansion} to \eqref{DerivedLambdaBCforL-1}, i.e. \eqref{LambdaInductiveStep} is equivalent to
\begin{multline}
\hat{\lambda}_{q_{n-1}}^{o}\left(t_{n}\right)\equiv\left[\begin{array}{c}
\lambda_{q_{n-1},0}^{o}\left(t_{n}\right)\\
\lambda_{q_{n-1}}^{o}\left(t_{n}\right)
\end{array}\right]
\\
=\left[\begin{array}{cccc}
1 & 0 & \cdots & 0\\
\frac{\partial c}{\partial x_{1}} & \frac{\partial\xi_{1}}{\partial x_{1}} & \cdots & \frac{\partial\xi_{n}}{\partial x_{1}}\\
\vdots & \vdots & \ddots & \vdots\\
\frac{\partial c}{\partial x_{n}} & \frac{\partial\xi_{1}}{\partial x_{n}} & \cdots & \frac{\partial\xi_{n}}{\partial x_{n}}
\end{array}\right]_{\sigma_n} \left[\begin{array}{c}
\lambda_{q_{n},0}^{o}\left(t_{n}+\right)\\
\lambda_{q_{n}}^{o}\left(t_{n}+\right)
\end{array}\right]+p\left[\begin{array}{c}
0\\
\nabla m
\end{array}\right].
\end{multline}

This gives
\begin{align}
\lambda_{q_{n-1},0}^{o}\left(t_{n}\right)&=1,
\\
\lambda_{q_{n-1}}^{o}\left(t_{n}\right)&=\nabla\xi^{T}\lambda_{q_{n}}^{o}\left(t_{n}+\right)+\nabla c_{\sigma_{n}}+p\nabla m_{q_{n-1}q_{n}}\,.
\label{DerivedLambdaBCforL-1Autonomous}
\end{align}

Differentiating \eqref{LambdaHatGeneralDefinition} with respect to $t$ leads to
\begin{equation}
\frac{d}{dt}\hat{\lambda}_{q_{n-1}}^{o}\left(t\right)=-\left(\frac{\partial\hat{f}_{q_{n-1}}}{\partial\hat{x}_{q_{n-1}}}\left(\hat{x}_{q_{n-1}}^{o}\left(t\right),u_{q_{n-1}}^{o}\left(t\right)\right)\right)^{T}\hat{\lambda}_{q_{n-1}}^{o}\left(t\right),
\end{equation}
which is equivalent to
\begin{align}
&\frac{d}{dt}\lambda_{q_{n-1},0}^{o}\left(t\right)=0, 
\\
&\frac{d}{dt}\lambda_{q_{n-1}}^{o}\left(t\right)=-\left(\frac{\partial l_{q_{n-1}}\left(x_{q_{n-1}}^{o},u_{q_{n-1}}^{o}\right)}{\partial x_{q_{n-1}}}\right)\lambda_{0}^{o}\left(t\right) 
\nonumber\\
&\hspace{55pt}  -\left(\frac{\partial f_{q_{n-1}}\left(x_{q_{n-1}}^{o},u_{q_{n-1}}^{o}\right)}{\partial x_{q_{n-1}}}\right)^{T}\lambda_{q_{n-1}}^{o}\left(t\right). &\label{LambdaDynamicsBeforeSwitchAutonomous}
\end{align}

Therefore, $\lambda_{q_{n-1},0}^{o}\left(t\right)=1$ for $t\in \left(t_{n-1},t_n\right)$ is obtained as before and 
\begin{equation}
\dot{\lambda}_{q_{n-1}}^{o}=-\frac{\partial H_{q_{n-1}}\left(x_{q_{n-1}}^{o},\lambda_{q_{n-1}}^{o},u_{q_{n-1}}^{o}\right)}{\partial x_{q_{n-1}}},
\end{equation}
holds for $t\in \left(t_{n-1},t_n\right)$ with the Hamiltonian defined as
\begin{multline}
H_{q_{n-1}}\left(x_{q_{n-1}},\lambda_{q_{n-1}},u_{q_{n-1}}\right)=l_{q_{n-1}}\left(x_{q_{n-1}},u_{q_{n-1}}\right)
\\
+\lambda_{q_{n-1}}^{T}f_{q_{n-1}}\left(x_{q_{n-1}},u_{q_{n-1}}\right).
\end{multline}

Also from \eqref{PreHamiltonianBeforeManySwitchings} the minimization of the Hamiltonian is concluded, i.e.
\begin{equation}
H_{q_{n-1}}\left(x_{q_{n-1}},\lambda_{q_{n-1}},u_{q_{n-1}}\right)\leq H_{q_{n-1}}\left(x_{q_{n-1}},\lambda_{q_{n-1}},v\right),
\end{equation}
for all $v\in U_{q_{n-1}}$. 

Evaluating both $H_{q_{n-1}}$ and $H_{q_n}$ at $t_n$ gives
\begin{multline}
\hspace{-6pt} H_{q_{n-1}}\left(t_{n}-\right)=l_{q_{n-1}}\left(x_{q_{n-1}}\left(t_{n}-\right),u_{q_{n-1}}\left(t_{n}-\right)\right)
\\
\hfill +\lambda_{q_{n-1}}\left(t_{n}-\right)^{T}f_{q_{n-1}}\left(x_{q_{n-1}}\left(t_{n}-\right),u_{q_{n-1}}\left(t_{n}-\right)\right)
\\
=\left[\hat{\lambda}_{q_{n-1}}^{o}\left(t_{n}-\right)\right]^{T}\hat{f}_{q_{n-1}}\left(\hat{x}_{q_{n-1}}\left(t_{n}-\right),u_{q_{n-1}}^{o}\left(t_{n}-\right)\right) \hfill
\\
=\left[\frac{\partial\hat{\xi}\left(\hat{x}_{q_{n-1}}^{o\left(t_{n}^{-}\right)}\right)}{\partial\hat{x}_{q_{n-1}}}^{T}\hat{\lambda}_{q_{n}}^{o\left(t_{n}^{+}\right)}+p_{n}\frac{\partial\hat{m}\left(\hat{x}_{q_{n-1}}^{o\left(t_{n}^{-}\right)}\right)}{\partial\hat{x}_{q_{n-1}}}\right]^{T}\hat{f}_{q_{n-1}}^{\left(t_{n}^{-}\right)} \hfill
\\
=\left[\hat{\lambda}_{q_{n}}^{o}{}_{\left(t_{n}^{+}\right)}^T\frac{\partial\hat{\xi}}{\partial\hat{x}_{q_{n-1}}}+\gamma_{n}\left[\hat{f}_{q_{n},\hat{\xi}_{\sigma_{n}}}^{\hat{\xi}_{\sigma_{n}},q_{n-1}}\right]^{T}\hat{\lambda}_{q_{n}}^{o\left(t_{n}^{+}\right)}\left[\frac{\partial\hat{m}_{q_{n-1}q_{n}}}{\partial\hat{x}_{q_{n-1}}}\right]^{T}\right]\hat{f}_{q_{n-1}}^{\left(t_{n}-\right)} \hfill
\\
=\hat{\lambda}_{q_{n}\left(t_{n}^{+}\right)}^{o\;T}\frac{\partial\hat{\xi}}{\partial\hat{x}_{q_{n-1}}}\hat{f}_{q_{n-1}}^{\left(t_{n}^{-}\right)}+\frac{\left[\hat{f}_{q_{n},\hat{\xi}_{\sigma_{n}}}^{\hat{\xi}_{\sigma_{n}},q_{n-1}}\right]^{T}\hat{\lambda}_{q_{n}}^{o\left(t_{n}^{+}\right)}}{\left[\frac{\partial\hat{m}_{q_{n-1}q_{n}}}{\partial\hat{x}_{q_{n-1}}}\right]^{T}\hat{f}_{q_{n-1}}^{\left(t_{n}^{-}\right)}}\left[\frac{\partial\hat{m}_{q_{n-1}q_{n}}}{\partial\hat{x}_{q_{n-1}}}\right]^{T}\hat{f}_{q_{n-1}}^{\left(t_{n}^{-}\right)} \hfill
\\
=\hat{\lambda}_{q_{n}}^{o}{}_{\left(t_{n}^{+}\right)}^{T}\frac{\partial\hat{\xi}_{\sigma_{n}}}{\partial\hat{x}_{q_{n-1}}}\hat{f}_{q_{n-1}}^{\left(t_{n}^{-}\right)}+\hat{\lambda}_{q_{n}}^{o}{}_{\left(t_{n}^{+}\right)}^{T}\left[\hat{f}_{q_{n}}^{\left(t_{n}\right)}-\frac{\partial\hat{\xi}_{\sigma_{n}}}{\partial\hat{x}_{q_{n-1}}}\hat{f}_{q_{n-1}}^{\left(t_{n}^{-}\right)}\right] \hfill
\\
=\hat{\lambda}_{q_{n}}^{o}\left(t_{n}+\right)^{T}\hat{f}_{q_{n}}\left(\hat{x}_{q_{n}}^{o}\left(t_{n}\right),u_{q_{n}}^{o}\left(t_{n}\right)\right) \hfill
\\
=l_{q_{n}}\left(x_{q_{n}}^{o}\left(t_{n}\right),u_{q_{n}}^{o}\left(t_{n}\right)\right)+\lambda_{q_{n}}^{o}\left(t_{n}+\right)^{T}f_{q_{n}}\left(x_{q_{n}}^{o}\left(t_{n}\right),u_{q_{n}}^{o}\left(t_{n}\right)\right) \hfill
\\
=H_{q_{n}}\left(t_{n}+\right),
\end{multline}
which is equivalent to \eqref{Hamiltonian jump}. This completes the proof of the Hybrid Minimum Principle.
\end{proof}

\section{Time-Varying Vector Fields, Costs, and Switching Manifolds}
\label{sec:HMPTV}
For simplicity of the notation, the statement and the proof of the Hybrid Minimum Principle in Theorem \ref{theorem:HMP} were presented for the case of time-invariant vector fields, time-invariant running, switching and terminal costs, and time-invariant switching manifolds. It is, however, no loss of generality as time-varying hybrid optimal control problems can be converted to time-invariant problems by the extension of states, vector fields, etc. as
\begin{equation}
\tilde{x}_{q}:=\left[\begin{array}{c}
\theta\\
\hat{x}_{q}
\end{array}\right]\equiv\left[\begin{array}{c}
\theta\\
z_{q}\\
x_{q}
\end{array}\right],
\end{equation}
resulting in augmented vector fields as
\begin{equation}
\dot{\tilde{x}}_{q}=\tilde{f}_{q}\left(\tilde{x}_{q},u_{q}\right):=\left[\begin{array}{c}
1\\
\hat{f}_{q}\left(\hat{x}_{q},u_{q},t\right)
\end{array}\right]\equiv\left[\begin{array}{c}
1\\
l_{q}\left(x,u,\theta\right)\\
f_{q}\left(x,u,\theta\right)
\end{array}\right],
\end{equation}
subject to the initial condition
\begin{equation}
\tilde{h}_{0}=\left(q_{0},\hat{x}_{q_{0}}\left(t_{0}\right)\right)=\left(q_{0},\left[\begin{array}{c}
t_{0}\\
0\\
x_{0}
\end{array}\right]\right),
\end{equation}
with the switching manifold
\begin{equation}
\tilde{m}\left(\tilde{x}\right) := m\left(x,\theta\right),
\end{equation}
and the extended jump
function defined as
\begin{equation}
\tilde{x}_{q_{j}}\left(t_{j}\right)=\tilde{\xi}_{\sigma_{j}}\left(\tilde{x}_{q_{j-1}}\left(t_{j}-\right)\right):=\left[\begin{array}{c}
\theta\left(t_{j}-\right)\\
z\left(t_{j}-\right)+c\left(x\left(t_{j}-\right),\theta\right)\\
\xi_{\sigma_{j}}\left(x\left(t_{j}-\right),\theta\right)
\end{array}\right].
\end{equation}

\begin{corollary}
For the case of time-varying vector fields, costs, and switching manifolds, the statement of Theorem \ref{theorem:HMP} holds with the definition of the family of system Hamiltonians as 
\begin{equation}
H_{q}\left(x_{q},\lambda_{q},u_{q},t\right)= l_{q}\left(x_{q},u_{q},t\right)+\lambda_{q}^{T}\, f_{q}\left(x_{q},u_{q},t\right),
\end{equation}
i.e. Equations \eqref{StateDynamics}-\eqref{HminWRTu} hold and the Hamiltonian
boundary conditions \eqref{Hamiltonian jump} become
\begin{multline}
H_{q_{j-1}}\left(t_{j}-\right)\equiv H_{q_{j-1}}\left(x_{q_{j-1}}^{o},\lambda_{q_{j-1}}^{o},u_{q_{j-1}}^{o},t_{j}^{o}-\right)
\\
=H_{q_{j}}\left(x_{q_{j}}^{o},\lambda_{q_{j}}^{o},u_{q_{j}}^{o},t_{j}^{o}+\right)+\frac{\partial c_{\sigma_{j}}}{\partial t}+p_{j}\frac{\partial m_{q_{j-1}q_{j}}}{\partial t}
\\
\equiv H_{q_{j}}\left(t_{j}+\right) +\frac{\partial c_{\sigma_{j}}}{\partial t}+p_{j}\frac{\partial m_{q_{j-1}q_{j}}}{\partial t}.
\end{multline}
\hfill $\square$
\end{corollary}


\section{Analytic Example}
\label{sec:AnalyticExample}
\subsection{Problem Statement} 
Consider a hybrid system with the indexed vector fields:
\begin{align}
\dot{x} &=f_{q_1}\left(x,u\right)=x+x\, u,\label{Ex1f_1}
\\
\dot{x} &=f_{q_2}\left(x,u\right)=-x+x\, u,\label{Ex1f_2}
\vspace{-1mm}
\end{align}
and the hybrid optimal control problem
\begin{multline}
J\left(t_{0},t_{f},h_{0},1;I_1\right)
=\int_{t_0}^{t_{s}}\frac{1}{2}u^{2}dt+\frac{1}{1+\left[x\left(t_{s}-\right)\right]^{2}}
\\
+\int_{t_{s}}^{t_{f}}\frac{1}{2}u^{2}dt+\frac{1}{2}\left[x\left(t_{f}\right)\right]^{2},
\label{Ex1TotalCost}
\vspace{-1mm}
\end{multline}
subject to the initial condition $h_{0}=\left(q\left(t_{0}\right),x\left(t_{0}\right)\right)=\left(q_1,x_{0}\right)$ provided at the initial time $t_{0}=0$. At the controlled switching instant $t_s$, the boundary condition for the continuous state is provided by the jump map $x\left(t_{s}\right) =\xi\left(x\left(t_{s}-\right)\right)=-x\left(t_{s}-\right)$.

\subsection{The HMP Formulation}
Writing down the Hybrid Minimum Principle results for the above HOCP, the Hamiltonians are formed as
\begin{align}
H_{q_1} &=\frac{1}{2}u^{2}+\lambda\, x\left(u+1\right),\label{Ex1H1}
\\
H_{q_2} &=\frac{1}{2}u^{2}+\lambda\, x\left(u-1\right),\label{Ex1H2}
\end{align}
from which the minimizing control input for both Hamiltonian functions is determined as
\begin{equation}
u^{o}=-\lambda x \,. \label{Ex1u^o}
\end{equation}

Therefore, the adjoint process dynamics, determined from \eqref{lambda dynamics} and with the substitution of the optimal control input from \eqref{Ex1u^o}, is written as
\begin{align}
\hspace{-2mm} \dot{\lambda}= \frac{-\partial H_{q_1}}{\partial x} = -\lambda\left(u^{o}+1\right)=\lambda\left(\lambda\, x-1\right), &{}  & t \in \left(t_0,t_s\right] , \label{Ex1Lambda1Dynamics}
\\
\hspace{-2mm} \dot{\lambda}=\frac{-\partial H_{q_2}}{\partial x} =-\lambda\left(u^{o}-1\right)=\lambda\left(\lambda\, x+1\right), &{}  & t \in \left(t_s,t_f\right] , \label{Ex1Lambda2Dynamics}
\end{align}
which are subject to the terminal and boundary conditions
\begin{align}
\lambda\left(t_{f}\right) &=\left.\nabla g\right|_{x\left(t_{f}\right)}=x\left(t_{f}\right),\label{Ex1Lambda(t_f)}
\\
\lambda\left(t_{s}-\right)\equiv\lambda\left(t_{s}\right) &=\left.\nabla\xi\right|_{x\left(t_{s}-\right)}\lambda\left(t_{s}+\right)+\left.\nabla c\right|_{x\left(t_{s}-\right)}
\notag\\
&=-\lambda\left(t_{s}+\right)+\frac{-2x\left(t_{s}-\right)}{\left(1+\left[x\left(t_{s}-\right)\right]^{2}\right)^{2}} \;.
\label{Ex1Lambda(t_s)}
\end{align}

The substitution of the optimal control input \eqref{Ex1u^o} in the continuous state dynamics \eqref{StateDynamics} gives
\begin{align}
\dot{x} &=\frac{\partial H_{q_1}}{\partial \lambda} =x\left(1+u^{o}\right)=-x\left(\lambda\, x-1\right), &{}  & t \in \left[t_0,t_s\right) ,\label{Ex1X1Dynamics}
\\
\dot{x} &=\frac{\partial H_{q_2}}{\partial \lambda}=x\left(-1+u^{o}\right)=-x\left(\lambda\, x+1\right), &{}  & t \in \left[t_s,t_f\right) ,\label{Ex1X2Dynamics}
\end{align}
which are subject to the initial and boundary conditions
\begin{align}
x\left(t_0\right) & =x\left(0\right) =x_{0},\label{Ex1X0}
\\
x\left(t_{s}\right) &=\xi\left(x\left(t_{s}-\right)\right)=-x\left(t_{s}-\right) .\label{Ex1Xs}
\end{align}

The Hamiltonian continuity condition \eqref{Hamiltonian jump} states that
\begin{multline}
H_{q_1}\left(t_{s}-\right)=\frac{1}{2}\left[u^{o}\left(t_{s}-\right)\right]^{2}+\lambda\left(t_{s}-\right)x\left(t_{s}-\right)\left[u^{o}\left(t_{s}-\right)+1\right]
\\
=\frac{1}{2}\left[-\lambda\left(t_{s}-\right)x\left(t_{s}-\right)\right]^{2}
+\lambda\left(t_{s}-\right)x\left(t_{s}-\right)\left[-\lambda\left(t_{s}-\right)x\left(t_{s}-\right)+1\right]
\\
= H_{q_2}\left(t_{s}+\right)
=\frac{1}{2}\left[u^{o}\left(t_{s}+\right)\right]^{2}+\lambda\left(t_{s}+\right)x\left(t_{s}+\right)\left[u^{o}\left(t_{s}+\right)-1\right]
\\
=\frac{1}{2}\left[-\lambda\left(t_{s}+\right)x\left(t_{s}+\right)\right]^{2}
+\lambda\left(t_{s}+\right)x\left(t_{s}+\right)\left[-\lambda\left(t_{s}+\right)x\left(t_{s}+\right)-1\right], \hspace{-3mm}
\end{multline}
which can be written, using \eqref{Ex1Xs}, as
\begin{equation}
x\left(t_{s^-}\right)\left[\lambda\left(t_{s^-}\right)-\lambda\left(t_{s^+}\right)\right]
=\frac{1}{2}\left[x\left(t_{s^-}\right)\right]^{2}\left[\left[\lambda\left(t_{s^-}\right)\right]^{2}-\left[\lambda\left(t_{s^+}\right)\right]^{2}\right] .
\label{Ex1Hcontinuity}
\end{equation}

\begin{figure}
\begin{center}
\includegraphics[width=.98\columnwidth]{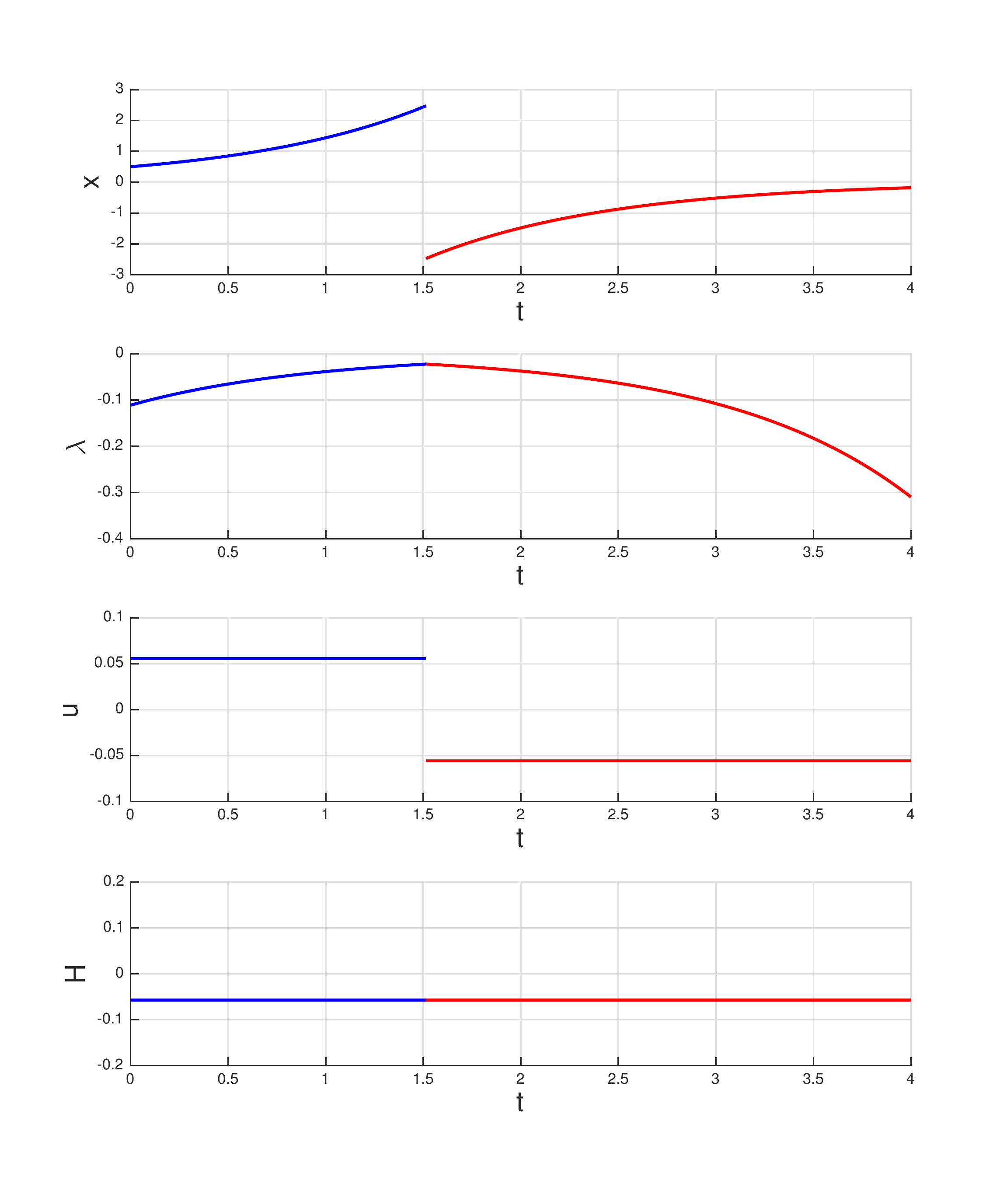}
\vspace{-20pt}
\caption{The optimal state and adjoint processes, optimal inputs and the Hamiltonians for the system in Example 1 with $x_{0}=0.5$ and $t_{f}=4$}
\end{center}
\label{Fig:Ex1IFAC2014}
\end{figure}

\subsection{The HMP Results}
The solution to the set of ODEs \eqref{Ex1Lambda1Dynamics}, \eqref{Ex1Lambda2Dynamics}, \eqref{Ex1X1Dynamics}, \eqref{Ex1X2Dynamics} together with the initial condition \eqref{Ex1X0} expressed at $t_0$, the terminal condition \eqref{Ex1Lambda(t_f)} determined at $t_f$ and the boundary conditions \eqref{Ex1Xs} and \eqref{Ex1Lambda(t_s)} provided at $t_s$ which is not a priori fixed but determined by the Hamiltonian continuity condition \eqref{Ex1Hcontinuity},  provides the optimal control input and its corresponding optimal trajectory that minimize the cost $J\left(t_{0},t_{f},h_{0},1;I_1\right)$ over $\bm{I_{1}}$, the family of hybrid inputs with one switching. The informativeness of the HMP results do not depend on the solution methodology for the ODE's and the associated boundary conditions. Interested readers are referred to \cite{APPECIFAC2014} for further analytic steps to reduce the above boundary value ODE problem into a set of algebraic equations using the special forms of the differential equations under study.

\section{Electric Vehicle with Transmission}

\label{sec:EVexample}

Consider the hybrid model of an electric vehicle equipped with a dual planetary transmission (presented in detail in \cite{APPEC2017NAHS}) with the hybrid automata diagram in Figure \ref{fig:HybridAutomata}. The discrete states $q_1, q_2, q_5, q_6$ correspond to fixed gear ratios while $q_3, q_4$ represent the system dynamics in transition between gears.

The set of vector fields $F$ is given as
\begin{align}
f_{q_1}\left(x,u\right) & =-A_1 x^2 + B_1 u -C_1 x - D_1 \,, \label{q1Vehicle}
\\
f_{q_2}\left(x,u\right) & =-A_2 x^2 + B_2 \frac{u}{x} - C_2 x - D_2 \,, \label{q2Vehicle}
\\
f_{q_3}^{\left(1\right)}\left(x,u\right) &= -A_{SS}x^{(1)}+A_{SR}x^{(2)}-A_{SA}\left(x^{(1)}+R_{2}x^{(2)}\right)^{2} \hfill  
\nonumber\\
& +B_{SM}^3 u^{\left(1\right)} +B_{SS}^3 u^{\left(2\right)} - B_{SR}^3 u^{\left(3\right)} -D_{SL} \,, 
\nonumber\\
f_{q_3}^{\left(2\right)}\left(x,u\right) &= A_{RS}x^{(1)}-A_{RR}x^{(2)}-A_{RA}\left(x^{(1)}+R_{2}x^{(2)}\right)^{2} \hfill 
\nonumber\\
& \hfill +B_{RM}^3 u^{\left(1\right)} - B_{RS}^3 u^{\left(2\right)} - B_{RR}^3 u^{\left(3\right)}-D_{RL} \,, 
\\
f_{q_4}^{\left(1\right)}\left(x,u\right)&= -A_{SS}x^{(1)}+A_{SR}x^{(2)}-A_{SA}\left(x^{(1)}+R_{2}x^{(2)}\right)^{2} \hfill 
\nonumber\\
& \hfill +B_{SM}\frac{u^{\left(1\right)}}{x^{(1)}+R_{1}x^{(2)}} +B_{SS} u^{\left(2\right)} - B_{SR} u^{\left(3\right)} -D_{SL} \,,
\nonumber\\
f_{q_4}^{\left(2\right)}\left(x,u\right)&= A_{RS}x^{(1)}-A_{RR}x^{(2)}-A_{RA}\left(x^{(1)}+R_{2}x^{(2)}\right)^{2} \hfill 
\nonumber\\
& \hspace{-40pt} +B_{RM} \left(1+R_{1}\right)\frac{u^{\left(1\right)}}{x^{(1)}+R_{1}x^{(2)}} - B_{RS} u^{\left(2\right)} + B_{RR} u^{\left(3\right)}-D_{RL} \,, \hspace{-4pt} \label{q4Vehicle}
\\
f_{q_5}\left(x,u\right) &= -A_5 x^2 + B_5 u -C_5 x - D_5 \,,
\\
f_{q_6}\left(x,u\right) &= -A_6 x^2 + B_6 \frac{u}{x} - C_6 x - D_6 \,, \label{q6Vehicle}
\end{align}
where $x_{q_1}, x_{q_2}, x_{q_5}, x_{q_6} \in \mathbb{R}$, $x_{q_3}, x_{q_4} \in \mathbb{R}^2$ are the continuous components of the hybrid state, with the notation $x_{q_i}^{(j)}$ used for denoting the $j^{\text{th}}$ component, and $u_{q_1}, u_{q_2}, u_{q_5}, u_{q_6}\in \left[-1,1\right] \subset \mathbb{R}$, $u_{q_3}, u_{q_4}\in \left[-1,1\right]^3 \subset \mathbb{R}^3$ are the continuous components of the hybrid input, with the coefficients on the right hand side of equations assumed to have deterministically known values.

Since the powertrain in the transition between gears operates with one more degree of freedom than fixed gear ratio modes, transitions to and from $q_3, q_4$ are accompanied by dimension changes. The set of jump transition maps $\Xi$ is identified by
\begin{align}
\xi_{q_{1}q_{2}}=\xi_{q_{2}q_{1}} &= \text{id}_{\mathbb{R}} \,, \label{EVstateJumpMapFirst}
\\
\xi_{q_{1}q_{3}}=\xi_{q_{2}q_{4}}&: \; x \rightarrow \left[\begin{array}{c} g_{tr}^1 x\\ 0 \end{array}\right] \,,
\\
\xi_{q_{3}q_{1}}=\xi_{q_{4}q_{2}}&: \; \left[\begin{array}{c} x^{(1)}\\ x^{(2)} \end{array}\right] \rightarrow \frac{x^{(1)}}{g_{tr}^1} \,,
\\
\xi_{q_{3}q_{4}}=\xi_{q_{4}q_{3}} &= \text{id}_{\mathbb{R}^2} \,,
\\
\xi_{q_{3}q_{5}}=\xi_{q_{4}q_{6}} &: \; \left[\begin{array}{c} x^{(1)}\\ x^{(2)} \end{array}\right] \rightarrow g_{tr}^2 x^{(2)} \,,
\\
\xi_{q_{5}q_{3}}=\xi_{q_{6}q_{4}} &: \; x  \rightarrow \left[\begin{array}{c} 0 \\ \frac{x}{g_{tr}^2} \end{array}\right] \,,
\\
\xi_{q_{5}q_{6}}=\xi_{q_{6}q_{5}} &= \text{id}_{\mathbb{R}} \,, \label{EVstateJumpMapLast}
\end{align}

While initiations of gear changing can be made freely (and therefore switchings to $q_3, q_4$ are controlled), the transitions back to a fixed gear mode require the full stop for one of the degrees of freedom. Moreover, switchings between torque-constrained and power-constrained modes occur when the motor speed reaches a certain value. The set of switching manifolds $\cal M$ for the autonomous switchings are given by
\begin{align}
m_{q_1 q_2} = m_{q_2 q_1} & \equiv x - k_1 = 0\,, \label{EVswitchingManifoldFirst}
\\
m_{q_3 q_1} = m_{q_4 q_2} &  \equiv x^{(2)} = 0 \,,
\\
m_{q_{3}q_{4}}=m_{q_{4}q_{3}} & \equiv x^{(1)}+R_{1}x^{(2)}-k_2=0 \,,
\\
m_{q_3 q_5} = m_{q_4 q_6} & \equiv x^{(1)}  = 0 \,,
\\
m_{q_5 q_6} = m_{q_6 q_5} & \equiv x - k_3 = 0\,, \label{EVswitchingManifoldLast}
\end{align}

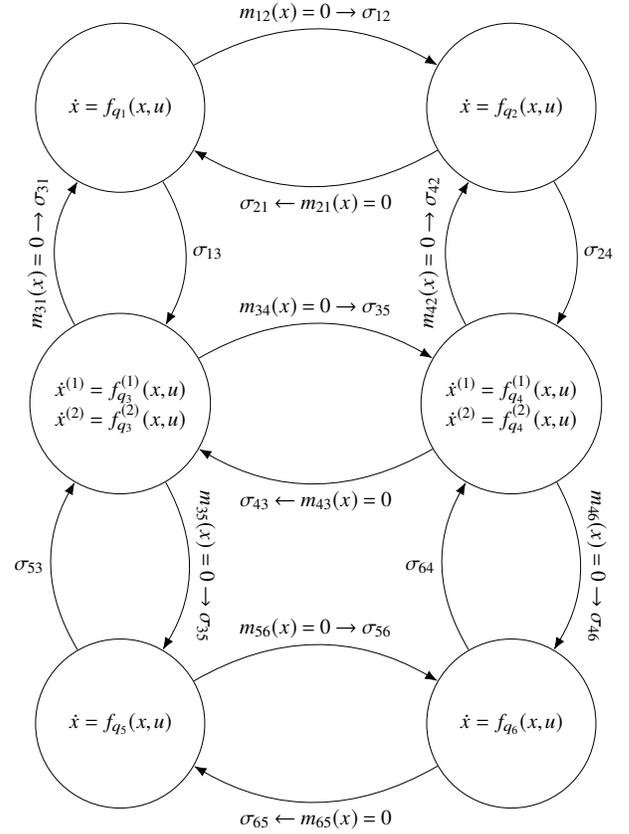
\begin{figure}
\centering
\scalebox{.8}{
\begin{tikzpicture}[node distance=65mm, auto]
  \node[state,minimum size=80pt] (s11)               {$\dot{x}=f_{q_1}(x,u)$};
  \node[state,minimum size=80pt] (s12) [right of=s11] {$\dot{x}=f_{q_2}(x,u)$};
 
  \node[state,minimum size=80pt](s21) [below=20mm of s11] {$\begin{array}{c} \dot{x}^{(1)}=f_{q_3}^{(1)}\left(x,u\right)\\ \dot{x}^{(2)}=f_{q_3}^{(2)}\left(x,u\right) \end{array}$};
  \node[state,minimum size=80pt]           (s22) [right of=s21]      {$\begin{array}{c} \dot{x}^{(1)}=f_{q_4}^{(1)}\left(x,u\right)\\ \dot{x}^{(2)}=f_{q_4}^{(2)}\left(x,u\right) \end{array}$};

  \node[state,minimum size=80pt] (s31) [below=24mm of s21] {$\dot{x}=f_{q_5}(x,u)$};
  \node[state,minimum size=80pt] (s32) [right of=s31]      {$\dot{x}=f_{q_6}(x,u)$};

 \path[-{Latex[scale=1.2]}]
            (s11) edge[bend left]  node {$m_{12}(x)=0 \rightarrow \sigma_{12}$} (s12)
            (s12) edge[bend left]  node {$\sigma_{21} \leftarrow m_{21}(x)=0$} (s11)
            
            (s11) edge[bend left]  node {$\sigma_{13}$} (s21)
            (s21) edge[bend left, sloped, anchor=center, above]  node {$m_{31}(x)=0 \rightarrow \sigma_{31}$} (s11)
            
            (s12) edge[bend left]  node {$\sigma_{24}$} (s22)
            (s22) edge[bend left, sloped, anchor=center, above]  node {$m_{42}(x)=0 \rightarrow \sigma_{42}$} (s12) 
            
            (s21) edge[bend left, sloped, anchor=center, above]  node {$m_{34}(x)=0 \rightarrow \sigma_{35}$} (s22)
            (s22) edge[bend left, sloped, anchor=center, below]  node {$\sigma_{43} \leftarrow m_{43}(x)=0$} (s21)
            
            (s31) edge[bend left, sloped, anchor=center, above]  node {$m_{56}(x)=0 \rightarrow \sigma_{56}$} (s32)
            (s32) edge[bend left, sloped, anchor=center, below]  node {$\sigma_{65} \leftarrow m_{65}(x)=0$} (s31)
            
            (s21) edge[bend left, sloped, anchor=center, above, rotate=179]  node {$m_{35}(x)=0 \rightarrow \sigma_{35}$} (s31)
            (s31) edge[bend left]  node {$\sigma_{53}$} (s21)
            
            (s22) edge[bend left, sloped, anchor=center, above, rotate=179]  node {$m_{46}(x)=0 \rightarrow \sigma_{46}$} (s32)
            (s32) edge[bend left]  node {$\sigma_{64}$} (s22);

\end{tikzpicture}
}

\caption{Hybrid Automata Diagram for an electric vehicle equipped with a dual planetary transmission} \label{fig:HybridAutomata}
\end{figure}

Now, consider the hybrid optimal control problem for the minimization of energy required for a certain task, e.g. starting from the stationary state, i.e. $h_0 \equiv \left(q,x\right)\left(t_0\right) = \left(q_1,0\right)$ with the sequence of discrete states given as $\left\{q_1, q_2, q_4, q_6\right\}$. Let the performance measure be determined by the sum of a terminal cost and the integral of the electric power consumption, i.e.
\begin{multline}
J\left(t_{0},t_{f},\left(q_{1},0\right),3;I_{3}\right)
=\int_{t_{0}}^{t_{s_{1}}}l_{q_{1}}\left(x,u\right)dt+\int_{t_{s_{1}}}^{t_{s_{2}}}l_{q_{2}}\left(x,u\right)dt
\\[-6pt]
+\int_{t_{s_{2}}}^{t_{s_{3}}}l_{q_{4}}\left(x,u\right)dt+\int_{t_{s_{3}}}^{t_{f}}l_{q_{6}}\left(x,u\right)dt + g\left(x\left(t_f\right)\right) \,,
\label{TotalEnergyEV}
\end{multline}
where the running costs $l_{q_i}$'s are the power consumption rates, determined from the motor efficiency map in \cite{APPEC2017NAHS} as
\begin{align}
l_{q_{1}}\left(x,u\right) &=a_{1}u^{2}+b_{1}xu+c_{1}u+d_{1}x\,,
\\
l_{q_{2}}\left(x,u\right) &=a_{2}\frac{u^{2}}{x^{2}}+b_{2}u+c_{2}\frac{u}{x}+d_{2}x \,,
\\
l_{q_{4}}\left(x,u\right) &=a_{4}\frac{\left(u^{\left(1\right)}\right)^{2}}{\left(x^{(1)}+R_{1}x^{(2)}\right)^{2}}+b_{4}u^{\left(1\right)}
\nonumber\\
&+c_{4}\frac{u^{\left(1\right)}}{x^{(1)}+R_{1}x^{(2)}}+d_{4}\left(x^{(1)}+R_{1}x^{(2)}\right)\,,
\\
l_{q_{6}}\left(x,u\right) &=a_{6}\frac{u^{2}}{x^{2}}+b_{6}u+c_{6}\frac{u}{x}+d_{6}x\,,
\\
g\left(x\left(t_f\right)\right) &= d_0 + d_1 x \left(t_f\right) + d_2 x\left(t_f\right)^2 \,.
\end{align}

\subsection*{The HMP Results:}

The Hybrid Minimum Principle in Section \ref{sec:HMP} can be used to identify the optimal input and the corresponding optimal trajectory for the performance measure \eqref{TotalEnergyEV}. Based on the HMP (details of the derivation are presented in the Appendix), 
optimal inputs are determined as
\begin{align}
&u_{q_{1}}^{o}\left(t\right)=\underset{_{\left[-1,1\right]}}{\text{sat}}\left(\frac{-\big(b_{1}x\left(t\right)+c_{1}+B_{1}\lambda\left(t\right)\big)}{2a_{1}}\right), \label{EVOptimalInputQ1}
\\
&u_{q_{2}}^{o}\left(t\right)=\underset{_{\left[-1,1\right]}}{\text{sat}}\left(\frac{-x\left(t\right)\big(b_{2}x\left(t\right)+c_{2}+B_{2}\lambda\left(t\right)\big)}{2a_{2}}\right),
\\
&\hspace{-6pt}\begin{array}{c}
u_{q_{4}}^{o\left(1\right)}\left(t\right)=\underset{_{\left[-1,1\right]}}{\text{sat}}\left(\frac{-\left(x_{\left(t\right)}^{\left(1\right)}+R_{1}x_{\left(t\right)}^{\left(2\right)}\right)\left[b_{4}\left(x_{\left(t\right)}^{\left(1\right)}+R_{1}x_{\left(t\right)}^{\left(2\right)}\right)+c_{4}+B_{SM}^{4}\lambda_{\left(t\right)}^{\left(1\right)}+B_{RM}^{4}\lambda_{\left(t\right)}^{\left(2\right)}\right]}{2a_{4}}\right),\hfill
\\
u_{q_{4}}^{o\left(2\right)}\left(t\right)=\left\{ \begin{array}{ccc}
-1 & \text{if} & B_{SS}\lambda^{\left(1\right)}\left(t\right)-B_{RS}\lambda^{\left(2\right)}\left(t\right)\geq0\\
0 & \text{if} & B_{SS}\lambda^{\left(1\right)}\left(t\right)-B_{RS}\lambda^{\left(2\right)}\left(t\right)<0
\end{array}\right.\,,\hfill
\\
u_{q_{4}}^{o\left(3\right)}\left(t\right)=\left\{ \begin{array}{ccc}
-1 & \text{if} & B_{RR}\lambda^{\left(2\right)}\left(t\right)-B_{SR}\lambda^{\left(1\right)}\left(t\right)\geq0\\
0 & \text{if} & B_{RR}\lambda^{\left(2\right)}\left(t\right)-B_{SR}\lambda^{\left(1\right)}\left(t\right)<0
\end{array}\right.\,,\hfill
\end{array} 
\\
&u_{q_{6}}^{o}\left(t\right)=\underset{_{\left[-1,1\right]}}{\text{sat}}\left(\frac{-x\left(t\right)\big(b_{6}x\left(t\right)+c_{6}+B_{6}\lambda\left(t\right)\big)}{2a_{6}}\right),\hfill \label{EVOptimalInputQ6}
\end{align}
where $x\left(t\right) \equiv x^o_{q_i}\left(t\right)$ are processes governed by the set of differential equations
\begin{align}
\dot{x}_{q_{1}}&=-A_{1}\left(x_{q_{1}}\left(t\right)\right)^{2}+B_{1}u_{q_{1}}^{o}\left(t\right)-C_{1}x_{q_{1}}\left(t\right)-D_{1}\,, \label{ExTrajectoryDynamicsQ1}
\\
\dot{x}_{q_{2}}&=-A_{2}\left(x_{q_{2}}\left(t\right)\right)^{2}+B_{2}\frac{u_{q_{2}}^{o}\left(t\right)}{x_{q_{2}}\left(t\right)}-C_{2}x_{q_{2}}\left(t\right)-D_{2}\,,\label{ExTrajectoryDynamicsQ2}
\\
\dot{x}_{q_{4}}^{\left(1\right)}&=-A_{SS}x_{q_{4}}^{(1)}\left(t\right)+A_{SR}x_{q_{4}}^{(2)}\left(t\right)-A_{SA}\left(x_{q_{4}}^{(1)}\left(t\right)+R_{2}x_{q_{4}}^{(2)}\left(t\right)\right)^{2}\nonumber
\\
& \hspace{-15pt} \hfill+B_{SM}\frac{u_{q_{4}}^{o\left(1\right)}\left(t\right)}{x_{q_{4}}^{(1)}+R_{1}x_{q_{4}}^{(2)}}+B_{SS}u_{q_{4}}^{o\left(2\right)}\left(t\right)-B_{SR}u_{q_{4}}^{o\left(3\right)}\left(t\right)-D_{SL}\,, \hspace{-10pt}
\\
\dot{x}_{q_{4}}^{\left(2\right)}&=A_{RS}x_{q_{4}}^{(1)}\left(t\right)-A_{RR}x_{q_{4}}^{(2)}\left(t\right)-A_{RA}\left(x_{q_{4}}^{(1)}\left(t\right)+R_{2}x_{q_{4}}^{(2)}\left(t\right)\right)^{2}\nonumber
\\
& \hspace{-16pt} +B_{RM}\left(1+R_{1}\right)\frac{u_{q_{4}}^{o\left(1\right)}\left(t\right)}{x_{q_{4}}^{(1)}+R_{1}x_{q_{4}}^{(2)}}-B_{RS}u_{q_{4}}^{o\left(2\right)}\left(t\right)+B_{RR}u_{q_{4}}^{o\left(3\right)}\left(t\right)-D_{RL}\,,
\\
\dot{x}_{q_{6}}&=-A_{6}\left(x_{q_{6}}\left(t\right)\right)^{2}+B_{6}\frac{u_{q_{6}}^{o}\left(t\right)}{x_{q_{6}}^{o}\left(t\right)}-C_{6}x_{q_{6}}^{o}\left(t\right)-D_{6}\,,\label{ExTrajectoryDynamicsQ6}
\end{align}
subject to the initial and boundary conditions:
\begin{align}
&x_{q_1}\left(t_0\right) = 0 \,, \vphantom{f_{f_{f_{f_f}}}} \label{ExICq1}
\\
&x_{q_2}\left(t_{s_1}\right) =x_{q_1}\left(t_{s_1}-\right) \,,
\\
&x_{q_4}\left(t_{s_2}\right) \equiv \left[\begin{array}{c}
x^{(1)}_{q_4}\left(t_{s_2}\right)\\
x^{(2)}_{q_4}\left(t_{s_2}\right)
\end{array}\right] =\left[\begin{array}{c}
g_{tr}^1 x_{q_2}\left(t_{s_2}-\right)\\
0
\end{array}\right] \,,
\\
&x_{q_6}\left(t_{s_3}\right) =g_{tr}^2 \, x_{q_4}^{(2)}\left(t_{s_3}-\right) \,, \label{ExICq6}
\end{align}
where the switching manifold condition are satisfied at the autonomous switching instances $t_{s_1}$, $t_{s_3}$, i.e.
\begin{align}
x_{q_1}\left(t_{s_1}-\right) &= k_1 \,, \label{ExSwManifoldT1}
\\
x_{q_4}^{(1)}\left(t_{s_3}-\right) &= 0 \,. \label{ExSwManifoldT3}
\end{align}
and $\lambda\left(t\right) \equiv \lambda^o_{q_i}\left(t\right)$ are backward processes governed by the set of differential equations
\begin{align}
\dot{\lambda}_{q_{6}}&=\frac{2a_{6}\left(u_{q_{6}}^{o}\left(t\right)\right)^{2}}{\left(x_{q_{6}}\left(t\right)\right)^{3}}+\frac{c_{6}u_{q_{6}}^{o}\left(t\right)}{\left(x_{q_{6}}\left(t\right)\right)^{2}}-d_{6} \hfill\nonumber
\\
&{\hspace{35pt}}+\lambda_{q_{6}}\left(t\right)\left(2A_{6}x_{q_{6}}\left(t\right)+B_{6}\frac{u_{q_{6}}^{o}\left(t\right)}{\left(x_{q_{6}}\left(t\right)\right)^{2}}+C_{6}\right)\,,
\label{ExAdjointDynamicsQ6}
\\
\dot{\lambda}_{q_{4}}^{(1)}&=\frac{2a_{4}\left(u_{q_{4}}^{o\left(1\right)}\left(t\right)\right)^{2}}{\left(x^{(1)}+R_{1}x^{(2)}\right)^{3}}+\frac{c_{4}u_{q_{4}}^{o\left(1\right)}\left(t\right)}{\left(x^{(1)}+R_{1}x^{(2)}\right)^{2}}-d_{4}\hfill\nonumber
\\
+&\lambda_{q_{4}}^{\left(1\right)}\left(A_{SS}+2A_{SA}\left(x^{(1)}+R_{2}x^{(2)}\right)+\frac{B_{SM}\left(1+R_{1}\right)u_{q_{4}}^{o\left(1\right)}}{\left(x^{(1)}+R_{1}x^{(2)}\right)^{2}}\right)\nonumber
\\
+&\lambda_{q_{4}}^{\left(2\right)}\left(-A_{RS}+2A_{RA}\left(x^{(1)}+R_{2}x^{(2)}\right)+\frac{B_{RM}\left(1+R_{1}\right)u_{q_{4}}^{o\left(1\right)}}{\left(x^{(1)}+R_{1}x^{(2)}\right)^{2}}\right)\,,
\\
\dot{\lambda}_{q_{4}}^{(2)}&=\left(\frac{2R_{1}a_{4}\left(u_{q_{4}}^{o\left(1\right)}\left(t\right)\right)^{2}}{\left(x^{(1)}+R_{1}x^{(2)}\right)^{3}}+\frac{R_{1}c_{4}u_{q_{4}}^{o\left(1\right)}\left(t\right)}{\left(x^{(1)}+R_{1}x^{(2)}\right)^{2}}-R_{1}d_{4}\right)\hfill\nonumber
\\
+&\lambda_{q_{4}}^{\left(1\right)}\left(-A_{SR}+2R_{2}A_{SA}\left(x^{(1)}+R_{2}x^{(2)}\right)+\frac{R_{1}B_{SM}\left(1+R_{1}\right)u_{q_{4}}^{o\left(1\right)}}{\left(x^{(1)}+R_{1}x^{(2)}\right)^{2}}\right)\hspace{8mm}\nonumber
\\
+&\lambda_{q_{4}}^{\left(2\right)}\left(A_{RR}+2R_{2}A_{RA}\left(x^{(1)}+R_{2}x^{(2)}\right)+\frac{R_{1}B_{RM}\left(1+R_{1}\right)u_{q_{4}}^{o\left(1\right)}}{\left(x^{(1)}+R_{1}x^{(2)}\right)^{2}}\right)\,,
\\
\dot{\lambda}_{q_{2}}&=\frac{2a_{2}\left(u_{q_{2}}^{o}\left(t\right)\right)^{2}}{\left(x_{q_{2}}\left(t\right)\right)^{3}}+\frac{c_{2}u_{q_{2}}^{o}\left(t\right)}{\left(x_{q_{2}}\left(t\right)\right)^{2}}-d_{2}\nonumber
\\
&{\hspace{40pt}}+\lambda_{q_{2}}\left(t\right)\left(2A_{2}x_{q_{2}}\left(t\right)+B_{2}\frac{u_{q_{2}}^{o}\left(t\right)}{\left(x_{q_{2}}\left(t\right)\right)^{2}}+C_{2}\right),
\\
\dot{\lambda}_{q_{1}}&=-b_{1}u_{q_{1}}^{o}\left(t\right)-d_{1}+\lambda_{q_{1}}\left(t\right)\left(2A_{1}x_{q_{1}}\left(t\right)+C_{1}\right),
\label{ExAdjointDynamicsQ1}
\end{align}
subject to the terminal and boundary conditions:
\begin{align}
\lambda_{q_6}\left(t_{f}\right)&= d_1 + 2 d_2 x_{q_6}\left(t_{f}\right) \,, \hfill \label{evExLambdaQ6}
\\
\lambda_{q_{4}}\left(t_{s_{3}}\right)&=\left[\begin{array}{c}
0\\
g_{tr}^{2}
\end{array}\right]\lambda_{q_{6}}\left(t_{s_{3}}+\right)+p_{3}\left[\begin{array}{c}
1\\
0
\end{array}\right]\,,\label{evExLambdaQ4}
\\
\lambda_{q_{2}}\left(t_{s_{2}}\right)&=\left[\begin{array}{cc}
g_{tr}^{1} & 0\end{array}\right]\left[\begin{array}{c}
\lambda_{q_{4}}^{\left(1\right)}\left(t_{s_{2}}+\right)\\
\lambda_{q_{4}}^{\left(2\right)}\left(t_{s_{2}}+\right)
\end{array}\right]=g_{tr}^{1}\lambda_{q_{4}}^{\left(1\right)}\left(t_{s_{2}}+\right)\,,\hfill\label{evExLambdaQ2}
\\
\lambda_{q_1}\left(t_{s_{1}}\right) &= \lambda_{q_2}\left(t_{s_{1}}+\right)+p_{1} \,. \hfill \label{evExLambdaQ1}
\end{align}

\begin{figure}
\vspace{-20pt}
\includegraphics[width=1.05\columnwidth]{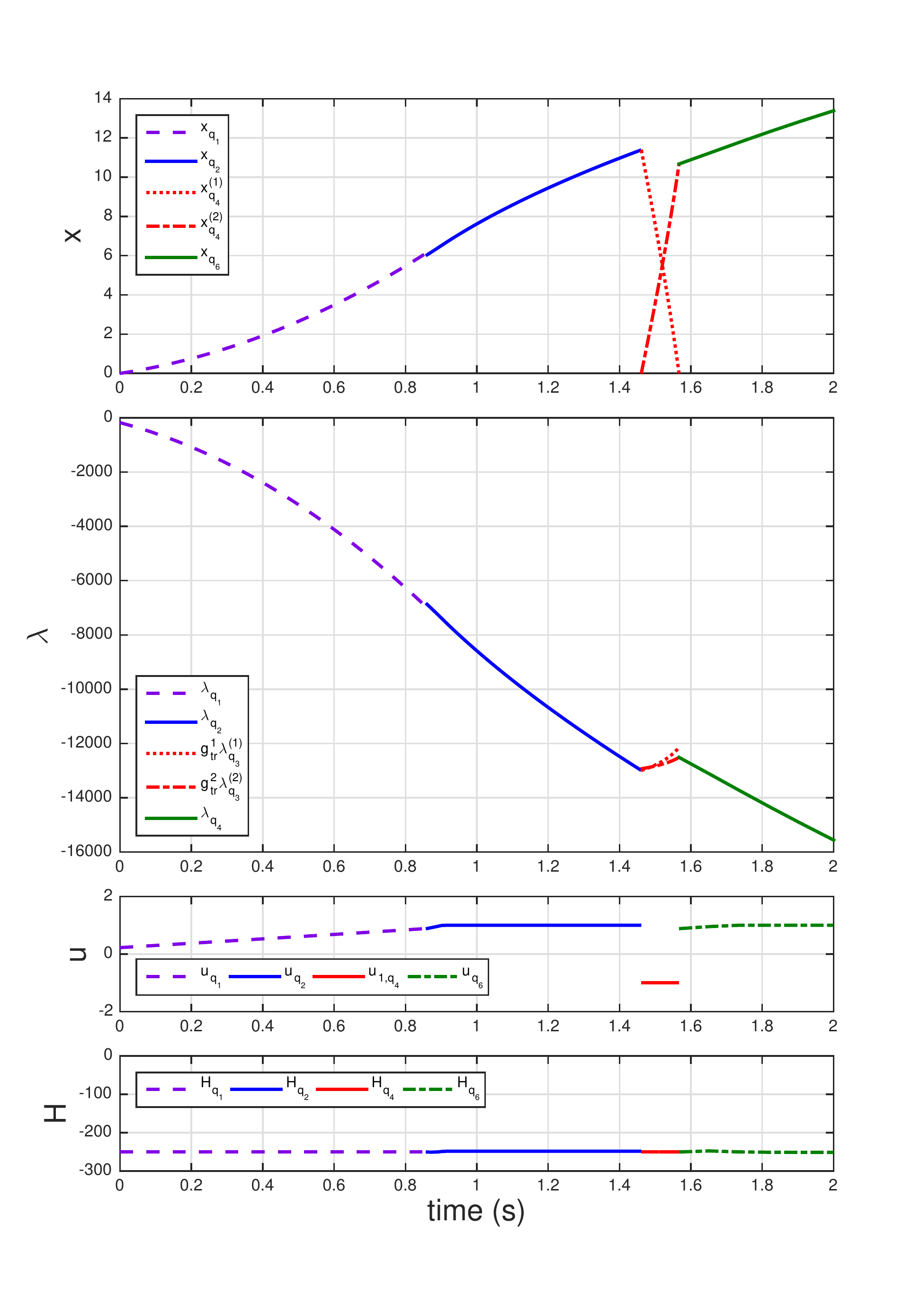}
\vspace*{-1.5cm}
\caption{Optimal state and adjoint processes, optimal input and the corresponding Hamiltonians for the example of vehicle with transmission.}
\label{Fig:AccelerationEnergyOptimal}
\end{figure}

\begin{figure}
\begin{center}
\includegraphics[width=.95\columnwidth]{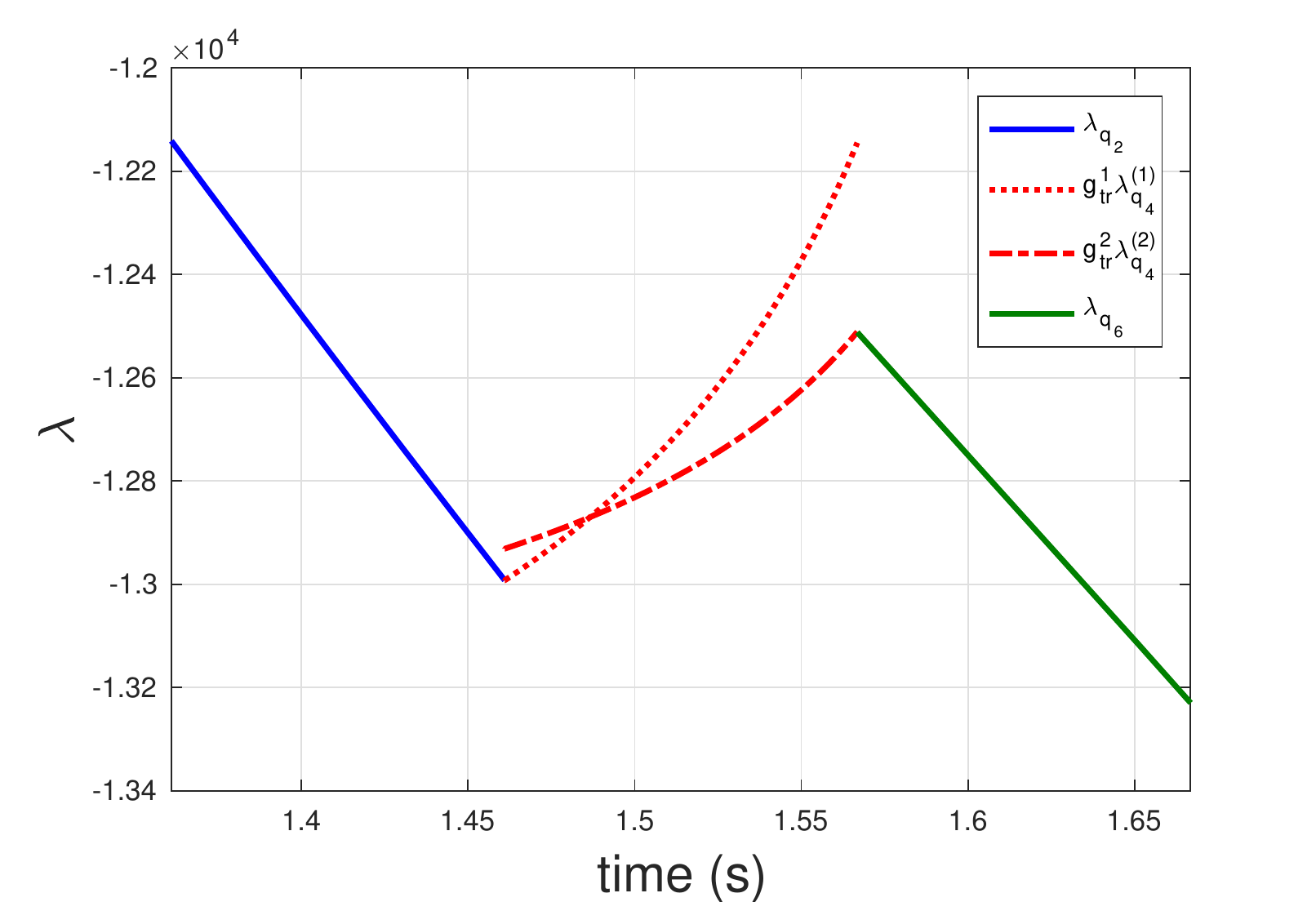}
\hspace{-15pt}
\vspace{-8pt}
\caption{Satisfaction of adjoint boundary conditions \eqref{evExLambdaQ4} and \eqref{evExLambdaQ2} which are accompanied by dimension-changes.}
\label{Fig:LambdaBC}
\end{center}
\vspace{-8pt}
\end{figure}

and the Hamiltonian boundary conditions at the optimal switching instances $t_{s_1}$, $t_{s_2}$, $t_{s_3}$, i.e.
\begin{multline}
\lambda_{q_{1}}\left(t_{s_{1}}^{-}\right)\left(-A_{1}\left(x_{q_{1}}\left(t_{s_{1}}^{-}\right)\right)^{2}+B_{1}u_{q_{1}}^{o}\left(t_{s_{1}}^{-}\right)-C_{1}x_{q_{1}}\left(t_{s_{1}}^{-}\right)-D_{1}\right)
\\
+a_{1}\left(u_{q_{1}}^{o}\left(t_{s_{1}}^{-}\right)\right)^{2}+b_{1}x_{q_{1}}\left(t_{s_{1}}^{-}\right)u_{q_{1}}^{o}\left(t_{s_{1}}^{-}\right)+c_{1}u_{q_{1}}^{o}\left(t_{s_{1}}^{-}\right)+d_{1}x_{q_{1}}\left(t_{s_{1}}^{-}\right)
\\
=\lambda_{q_{2}}\left(t_{s_{1}}^{+}\right)\left(-A_{2}\left(x_{q_{2}}\left(t_{s_{1}}^{+}\right)\right)^{2}+B_{2}\frac{u_{q_{2}}^{o}\left(t_{s_{1}}^{+}\right)}{x_{q_{2}}\left(t_{s_{1}}^{+}\right)}-C_{2}x_{q_{2}}\left(t_{s_{1}}^{+}\right)-D_{2}\right)
\\
+a_{2}\left(\frac{u_{q_{2}}^{o}\left(t_{s_{1}}^{+}\right)}{x_{q_{2}}\left(t_{s_{1}}^{+}\right)}\right)^{2}+b_{2}u_{q_{2}}^{o}\left(t_{s_{1}}^{+}\right)+c_{2}\frac{u_{q_{2}}^{o}\left(t_{s_{1}}^{+}\right)}{x_{q_{2}}\left(t_{s_{1}}^{+}\right)}+d_{2}x_{q_{2}}\left(t_{s_{1}}^{+}\right),  \label{ExHamiltonianT1}
\end{multline}
\begin{multline}
\lambda_{q_{2}}\left(t_{s_{2}}^{-}\right)\left(-A_{2}\left(x_{q_{2}}\left(t_{s_{2}}^{-}\right)\right)^{2}+B_{2}\frac{u_{q_{2}}^{o}\left(t_{s_{2}}^{-}\right)}{x_{q_{2}}\left(t_{s_{2}}^{-}\right)}-C_{2}x_{q_{2}}\left(t_{s_{2}}^{-}\right)-D_{2}\right)\hfill
\\
+a_{2}\left(\frac{u_{q_{2}}^{o}\left(t_{s_{2}}^{-}\right)}{x_{q_{2}}\left(t_{s_{2}}^{-}\right)}\right)^{2}+b_{2}u_{q_{2}}^{o}\left(t_{s_{2}}^{-}\right)+c_{2}\frac{u_{q_{2}}^{o}\left(t_{s_{2}}^{-}\right)}{x_{q_{2}}\left(t_{s_{2}}^{-}\right)}+d_{2}x_{q_{2}}\left(t_{s_{2}}^{-}\right),
\\
=-\lambda_{q_{4}}^{(1)}\left(t_{s_{2}}^{+}\right)\bigg(A_{SS}x_{q_{4}}^{(1)}\left(t_{s_{2}}^{+}\right)+A_{SA}\left(x_{q_{4}}^{(1)}\left(t_{s_{2}}^{+}\right)\right)^{2} \hfill
\\
+B_{SM}^{4}\frac{u_{q_{4}}^{o(1)}\left(t_{s_{2}}^{+}\right)}{x_{q_{4}}^{(1)}\left(t_{s_{2}}^{+}\right)}+B_{SS}^{4}u_{q_{4}}^{o(2)}\left(t_{s_{2}}^{+}\right)-B_{SR}^{4}u_{q_{4}}^{o(3)}\left(t_{s_{2}}^{+}\right)-D_{SL}\bigg)
\\
+\lambda_{q_{4}}^{(2)}\left(t_{s_{2}}^{+}\right)\bigg(A_{RS}x_{q_{4}}^{(1)}\left(t_{s_{2}}^{+}\right)-A_{RA}\left(x_{q_{4}}^{(1)}\left(t_{s_{2}}^{+}\right)\right)^{2}
\\
+B_{RM}^{4}\frac{u_{q_{4}}^{o(1)}\left(t_{s_{2}}^{+}\right)}{x_{q_{4}}^{(1)}\left(t_{s_{2}}^{+}\right)}-B_{RS}^{4}u_{q_{4}}^{o(2)}\left(t_{s_{2}}^{+}\right)+B_{RR}^{4}u_{q_{4}}^{o(3)}\left(t_{s_{2}}^{+}\right)-D_{RL}\bigg)
\\
+a_{4}\frac{\left(u_{q_{4}}^{o(1)}\left(t_{s_{2}}^{+}\right)\right)^{2}}{\left(x_{q_{4}}^{(1)}\left(t_{s_{2}}^{+}\right)\right)^{2}}+b_{4}u_{q_{4}}^{o(1)}\left(t_{s_{2}}^{+}\right)+c_{4}\frac{u_{q_{4}}^{o(1)}\left(t_{s_{2}}^{+}\right)}{x_{q_{4}}^{(1)}\left(t_{s_{2}}^{+}\right)}+d_{4}\left(x_{q_{4}}^{(1)}\left(t_{s_{2}}^{+}\right)\right) \,,
\end{multline}
\begin{multline}
\lambda_{q_{4}}^{(1)}\left(t_{s_{3}}^{-}\right)\bigg(A_{SR}x_{q_{4}}^{(2)}\left(t_{s_{3}}^{-}\right)-A_{SA}\left(R_{2}x_{q_{4}}^{(2)}\left(t_{s_{3}}^{-}\right)\right)^{2}\\+B_{SM}^{4}\frac{u_{q_{4}}^{o\left(1\right)}\left(t_{s_{3}}^{-}\right)}{R_{1}x_{q_{4}}^{(2)}\left(t_{s_{3}}^{-}\right)}+B_{SS}^{4}u_{q_{4}}^{o\left(2\right)}\left(t_{s_{3}}^{-}\right)-B_{SR}^{4}u_{q_{4}}^{o\left(3\right)}\left(t_{s_{3}}^{-}\right)-D_{SL}\bigg)\\-\lambda_{q_{4}}^{\left(2\right)}\left(t_{s_{3}}^{-}\right)\bigg(A_{RR}x_{q_{4}}^{(2)}\left(t_{s_{3}}^{-}\right)+A_{RA}\left(R_{2}x_{q_{4}}^{(2)}\left(t_{s_{3}}^{-}\right)\right)^{2}\\+B_{RM}^{4}\frac{u_{q_{4}}^{o\left(1\right)}\left(t_{s_{3}}^{-}\right)}{R_{1}x_{q_{4}}^{(2)}}-B_{RS}^{4}u_{q_{4}}^{o\left(2\right)}\left(t_{s_{3}}^{-}\right)+B_{RR}^{4}u_{q_{4}}^{o\left(3\right)}\left(t_{s_{3}}^{-}\right)-D_{RL}\bigg)\\+a_{4}\frac{\left(u_{q_{4}}^{o\left(1\right)}\left(t_{s_{3}}^{-}\right)\right)^{2}}{\left(R_{1}x_{q_{4}}^{(2)}\left(t_{s_{3}}^{-}\right)\right)^{2}}+b_{4}u_{q_{4}}^{o\left(1\right)}\left(t_{s_{3}}^{-}\right)+c_{4}\frac{u_{q_{4}}^{o\left(1\right)}\left(t_{s_{3}}^{-}\right)}{R_{1}x_{q_{4}}^{(2)}\left(t_{s_{3}}^{-}\right)}+d_{4}R_{1}x_{q_{4}}^{(2)}\left(t_{s_{3}}^{-}\right)
\\
=\lambda_{q_{6}}\left(t_{s_{3}}^{+}\right)\left(-A_{6}\left(x_{q_{6}}\left(t_{s_{3}}^{+}\right)\right)^{2}+B_{6}\frac{u_{q_{6}}^{o}\left(t_{s_{3}}^{+}\right)}{x_{q_{6}}\left(t_{s_{3}}^{+}\right)}-C_{6}x_{q_{6}}\left(t_{s_{3}}^{+}\right)-D_{6}\right)
\\
\hspace{-12pt} +a_{6}\left(\frac{u_{q_{6}}^{o}\left(t_{s_{3}}^{+}\right)}{x_{q_{6}}\left(t_{s_{3}}^{+}\right)}\right)^{2}+b_{6}u_{q_{6}}^{o}\left(t_{s_{3}}^{+}\right)+c_{6}\frac{u_{q_{6}}^{o}\left(t_{s_{3}}^{+}\right)}{x_{q_{6}}\left(t_{s_{3}}^{+}\right)}+d_{6}x_{q_{6}}\left(t_{s_{3}}^{+}\right)\,. \label{ExHamiltonianT3} \hspace{-5pt}
\end{multline}

For the 10 (scalar) ordinary differential equations \eqref{ExTrajectoryDynamicsQ1}--\eqref{ExTrajectoryDynamicsQ6}, \eqref{ExAdjointDynamicsQ6}--\eqref{ExAdjointDynamicsQ1}, the 3 a priori unknown switching instances $t_{s_1}$, $t_{s_2}$, $t_{s_3}$ and the 2 auxiliary introduced unknowns $p_1$, $p_3$ there are 15 equations provided by \eqref{ExICq1}--\eqref{ExICq6}, \eqref{evExLambdaQ6}--\eqref{evExLambdaQ1}, \eqref{ExHamiltonianT1}--\eqref{ExHamiltonianT3} in the form of initial, boundary and terminal conditions. 
It is not difficult to show that for the parameter values in \cite{APPEC2017NAHS, APPEC1ADHS2015, MechMachThry2015, USPatent2017}, the necessary optimality conditions of the HMP in the form of the above set of multiple-point boundary value differential equations uniquely identify optimal inputs and the corresponding optimal trajectories. The results are illustrated in Figure \ref{Fig:AccelerationEnergyOptimal}. In order to illustrate the satisfaction of the adjoint boundary conditions  \eqref{evExLambdaQ4} and \eqref{evExLambdaQ2}, the components $\lambda^{(1)}$ and $\lambda^{(2)}$ of the adjoint process in $t \in \left[t_{s_2},t_{s_3}\right]$ are multiplied  by $g_{tr}^1$ and $g_{tr}^2$ respectively, and are in-zoomed in Figure \ref{Fig:LambdaBC}. 
A more detailed derivation is provided in the Appendix. Interested readers are referred to \cite{APPEC2017NAHS} for more details and discussion on the results.




\bibliographystyle{IEEEtran}
\bibliography{AliPakniyatThesis}

\begin{appendices}

\renewcommand{\theequation}{\Alph{section}.\arabic{equation}}

\section{Derivation of the multiple-point boundary value differential equations}
\label{sec:Appendix}

In order to find the infimum of the hybrid cost \eqref{TotalEnergyEV} subject to the dynamics \eqref{q1Vehicle}--\eqref{q6Vehicle}, jump maps \eqref{EVstateJumpMapFirst}--\eqref{EVstateJumpMapLast}, and switching manifolds \eqref{EVswitchingManifoldFirst}--\eqref{EVswitchingManifoldLast} the following steps are taken.
\subsection{Formation of the Hamiltonians:}
The  family of system Hamiltonians are formed as
\begin{align}
&H_{q_1}\left(x,\lambda,u\right)= \lambda\left(-A_1 x^2 + B_1 u -C_1 x - D_1 \right) \label{EVHamiltonianQ1}
\nonumber\\
& \hspace{100pt} +a_{1}u^{2}+b_{1}xu+c_{1}u+d_{1}x \,,
\\
&H_{q_2}\left(x,\lambda,u\right)= \lambda\left(-A_2 x^2 + B_2 \frac{u}{x} - C_2 x - D_2\right)
\nonumber\\
& \hspace{100pt} +a_{2}\frac{u^{2}}{x^{2}}+b_{2}u+c_{2}\frac{u}{x}+d_{2}x \,,
\\
&H_{q_{4}}\left(x,\lambda,u\right)=\lambda^{\left(1\right)}\bigg(-A_{SS}x^{(1)}+A_{SR}x^{(2)}-A_{SA}\left(x^{(1)}+R_{2}x^{(2)}\right)^{2}\nonumber\\&\hspace{45pt}+B_{SM}^{4}\frac{u^{\left(1\right)}}{x^{(1)}+R_{1}x^{(2)}}+B_{SS}^{4}u^{\left(2\right)}-B_{SR}^{4}u^{\left(3\right)}-D_{SL}\bigg)\nonumber\\&\hspace{47pt}+\lambda^{\left(2\right)}\bigg(A_{RS}x^{(1)}-A_{RR}x^{(2)}-A_{RA}\left(x^{(1)}+R_{2}x^{(2)}\right)^{2}\nonumber\\&\hspace{47pt}+B_{RM}^{4}\frac{u^{\left(1\right)}}{x^{(1)}+R_{1}x^{(2)}}-B_{RS}^{4}u^{\left(2\right)}+B_{RR}^{4}u^{\left(3\right)}-D_{RL}\bigg)\nonumber\\&\hspace{47pt}+a_{4}\frac{\left(u^{\left(1\right)}\right)^{2}}{\left(x^{(1)}+R_{1}x^{(2)}\right)^{2}}+b_{4}u^{\left(1\right)}\nonumber\\&\hspace{47pt}+c_{4}\frac{u^{\left(1\right)}}{x^{(1)}+R_{1}x^{(2)}}+d_{4}\left(x^{(1)}+R_{1}x^{(2)}\right)\,,
\\
&H_{q_6}\left(x,\lambda,u\right)=\lambda\left(-A_6 x^2 + B_6 \frac{u}{x} - C_6 x - D_6\right)
\nonumber\\
& \hspace{100pt} +a_{6}\frac{u^{2}}{x^{2}}+b_{6}u+c_{6}\frac{u}{x}+d_{6}x \,. \label{EVHamiltonianQ6}
\end{align}

\subsection{Hamiltonian Minimization:}
The Hamiltonian minimization condition \eqref{HminWRTu} for the Hamiltonians \eqref{EVHamiltonianQ1}--\eqref{EVHamiltonianQ6} result in \eqref{EVOptimalInputQ1}--\eqref{EVOptimalInputQ6}.

\subsection{Continuous State Evolution:}
Taking the partial derivative of the Hamiltonians \eqref{EVHamiltonianQ1}--\eqref{EVHamiltonianQ6} with respect to $\lambda$, the continuous state dynamics \eqref{StateDynamics} are derived as \eqref{ExTrajectoryDynamicsQ1}--\eqref{ExTrajectoryDynamicsQ6}. These ODEs are subject to the initial and boundary conditions \eqref{StateIC} and \eqref{StateBC}, which for this problem are explicitly expressed in \eqref{ExICq1}--\eqref{ExICq6}.
By problem definition, the switchings from $q_1$ to $q_2$ and from $q_4$ to $q_6$ are autonomous, and therefore subject to the switching manifold conditions \eqref{ExSwManifoldT1} and \eqref{ExSwManifoldT3}.

\subsection{Evolution of the Adjoint Process:}
Taking the partial derivative of the Hamiltonians \eqref{EVHamiltonianQ1}--\eqref{EVHamiltonianQ6} with respect to $\lambda$, the adjoint process dynamics \eqref{lambda dynamics} are derived as \eqref{ExAdjointDynamicsQ6}--\eqref{ExAdjointDynamicsQ1}, which are 
subject to the terminal and boundary conditions \eqref{LambdaTerminal} and \eqref{LambdaBC}, which in this example are written as \eqref{evExLambdaQ6}--\eqref{evExLambdaQ1}.

\subsection{Boundary Conditions on Hamiltonians:}
Furthermore, the Hamiltonian continuity at switching instants \eqref{Hamiltonian jump} are expressed as
\begin{align}
H_{q_4}\left(x,\lambda,u\right)_{\left(t_{s_{3}}-\right)} &=H_{q_6}\left(x,\lambda,u\right)_{\left(t_{s_{3}}+\right)} \,,
\\
H_{q_2}\left(x,\lambda,u\right)_{\left(t_{s_{2}}-\right)} &=H_{q_4}\left(x,\lambda,u\right)_{\left(t_{s_{2}}+\right)} \,,
\\
H_{q_1}\left(x,\lambda,u\right)_{\left(t_{s_{1}}-\right)} &=H_{q_2}\left(x,\lambda,u\right)_{\left(t_{s_{1}}+\right)} \,,
\end{align}
which, with the evaluation of the Hamiltonians \eqref{EVHamiltonianQ1}--\eqref{EVHamiltonianQ6} at the switching instants, result in \eqref{ExHamiltonianT1}--\eqref{ExHamiltonianT3}.

\end{appendices}

\raggedbottom

\pagebreak

\balance

\begin{IEEEbiography}[{\includegraphics[width=1in,height=1.25in,clip,keepaspectratio]{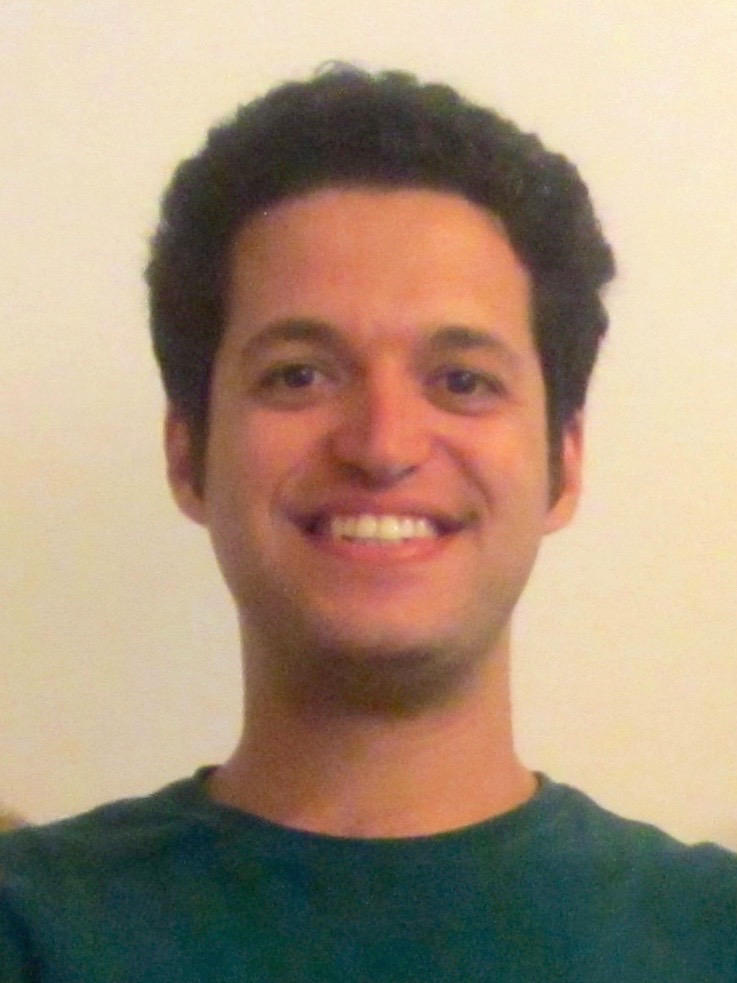}}]%
{Ali Pakniyat}
 received the B.Sc. degree in mechanical engineering from Shiraz University, Shiraz, Iran, in 2008, and the M.Sc. degree in mechanical engineering (applied mechanics and design) from Sharif University of Technology, Tehran, Iran, in 2010, and the Ph.D. degree in electrical engineering from McGill University, Montreal, QC, Canada, in September 2016, under the supervision of P. E. Caines.
He is currently a postdoctoral research fellow in the Department of Mechanical Engineering, University of Michigan, Ann Arbor, USA. 

He is a member of Robotics and Optimization for the Analysis of Human Motion (ROAHM) Lab and a former member of McGill Centre for Intelligent Machines (CIM) and Groupe d'\'Etudes et de Recherche en Analyse des D\'ecisions (GERAD). He serves as the 2018 Chair of the IEEE Southeast Michigan (SEM) Chapter 12: Control Systems Society.
His research interests include \text{deterministic} and stochastic optimal control, \text{nonlinear} and hybrid systems, analytical mechanics and chaos, with applications in the  automotive field, mathematical finance, sensors and actuators, and robotics.
\end{IEEEbiography}
\begin{IEEEbiography}[{\includegraphics[width=1in,height=1.25in,clip,keepaspectratio]{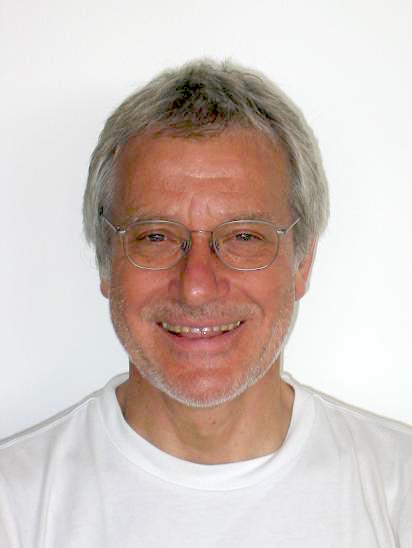}}]%
{Peter E. Caines}
received the BA in mathematics from Oxford University in 1967 and the PhD in systems and control theory in 1970 from Imperial College, University of London, under the supervision of David Q. Mayne, FRS. After periods as a postdoctoral researcher and faculty member at UMIST, Stanford, UC Berkeley, Toronto and Harvard, he joined McGill University, Montreal, in 1980, where he is James McGill Professor and Macdonald Chair in the Department of Electrical and Computer Engineering. In 2000 the adaptive control paper he coauthored with G. C. Goodwin and P. J. Ramadge (IEEE Transactions on Automatic Control, 1980) was recognized by the IEEE Control Systems Society as one of the 25 seminal control theory papers of the 20th century. He is a Life Fellow of the IEEE, and a Fellow of SIAM, the Institute of Mathematics and its Applications (UK) and the Canadian Institute for Advanced Research and is a member of Professional Engineers Ontario. He was elected to the Royal Society of Canada in 2003. In 2009 he received the IEEE Control Systems Society Bode Lecture Prize and in 2012 a Queen Elizabeth II Diamond Jubilee Medal. Peter Caines is the author of Linear Stochastic Systems, John Wiley, 1988, which is to be republished in 2018 as a SIAM Classic, and is a Senior Editor of Nonlinear Analysis – Hybrid Systems; his research interests include stochastic, mean field game, decentralized and hybrid systems theory,  together with their applications in a range of fields.
\end{IEEEbiography}

\end{document}